\documentclass[11pt,twoside, reqno]{article}

\usepackage{amssymb}
\usepackage{amsmath}
\usepackage{mathrsfs}
\usepackage{amsthm}
\usepackage{txfonts}

\usepackage[colorlinks=true, pdfstartview=FitV, linkcolor=blue, citecolor=blue, urlcolor=blue]{hyperref}

\allowdisplaybreaks

\def\rr{{\mathbb R}}
\def\rn{{{\rr}^n}}
\def\zz{{\mathbb Z}}

\def\nn{{\mathbb N}}
\def\fz{\infty}

\def\cs{{\mathcal S}}

\renewcommand\tilde{\widetilde}

\def\supp{{\rm{\,supp\,}}}

\def\ls{\lesssim}

\def\hs{\hspace{0.3cm}}

\def\r{\right}
\def\lf{\left}

\def\bint{{\ifinner\rlap{\bf\kern.30em--}
\int\else\rlap{\bf\kern.35em--}\int\fi}\ignorespaces}

\def\sbint{{\ifinner\rlap{\bf\kern.32em--}
\hspace{0.078cm}\int\else\rlap{\bf\kern.45em--}\int\fi}\ignorespaces}

\newtheorem{theorem}{Theorem}[section]
\newtheorem{lemma}[theorem]{Lemma}
\newtheorem{corollary}[theorem]{Corollary}

\theoremstyle{definition}
\newtheorem{remark}[theorem]{Remark}
\newtheorem{definition}[theorem]{Definition}
\newtheorem{assumption}[theorem]{Assumption}
\numberwithin{equation}{section}

\numberwithin{equation}{section}

\numberwithin{equation}{section}

\begin{document}

\arraycolsep=1pt

\title{\Large\bf Estimates for Schr\"{o}dinger Groups and Imaginary Power Operators on Weak
Hardy Spaces Associated with Non-negative Self-adjoint Operators and Ball Quasi-Banach Function Spaces
\footnotetext{\hspace{-0.35cm}
{\it 2020 Mathematics Subject Classification}.
{Primary 42B35; Secondary 42B30, 42B25.}
\endgraf{\it Key words and phrases.} Hardy space, ball quasi-Banach function space, self-adjoint operator, Schr\"{o}dinger group, imaginary power operator.
}}
\author{Xiong Liu and Wenhua Wang}
\date{  }
\maketitle

\vspace{-0.8cm}

\begin{center}
\begin{minipage}{13cm}\small
{\noindent{\bf Abstract.}
Let $(\mathbb{X},d,\mu)$ be a doubling metric measure space, $L$ a non-negative self-adjoint operator on $L^2(\mathbb{X})$
satisfying the Davies-Gaffney estimate, and $X(\mathbb{X})$ a ball quasi-Banach function space on $\mathbb{X}$ satisfying some mild assumptions
with $p\in(0,\infty)$ and $s_0\in(0,\min\{p,1\}]$.
In this article, the authors study the weak Hardy space $WH_{X,L}(\mathbb{X})$ associated with $L$ and $X(\mathbb{X})$,
and then give the atomic and molecular decompositions of $WH_{X,L}(\mathbb{X})$.
As applications, the authors establish the boundedness estimate of Schr\"{o}dinger groups for fractional powers of $L$ on $WH_{X,L}(\mathbb{X})$:
$$\left\|(I+L)^{-\beta/2}e^{i\tau L^{\gamma/2}}f\right\|_{WH_{X,L}(\mathbb{X})}\leq C\left(1+|\tau|\right)^{n(\frac{1}{s_0}-\frac{r}{2})}\|f\|_{WH_{X,L}(\mathbb{X})},$$
where $0<\gamma\neq1$, $\beta\in[\gamma n(\frac{1}{s_0}-\frac{1}{2}),\infty)$,
$r\in(0,1]$, $\tau\in \mathbb{R}$, and $C>0$ is a constant.
Moreover, when $(\mathbb{X},d,\mu)$ is an Ahlfors $n$-regular metric measure space and $L$ satisfies the Gaussian upper bound estimate,
the authors also obtain the boundedness estimate of imaginary power operators of $L$ on $WH_{X,L}(\mathbb{X})$:
$$\left\|L^{i\tau}f\right\|_{WH_{X,L}(\mathbb{X})}\leq C\left(1+|\tau|\right)^{n(\frac{1}{s_0}-\frac{r}{2})}\|f\|_{WH_{X,L}(\mathbb{X})},$$
where $\alpha>n(\frac{1}{s_0}-\frac{1}{2})$, $r\in(\frac{n/s_0}{\alpha+n/2},1]$, $\tau\in \mathbb{R}$, and $C>0$ is a constant. These results are also novelty for strong Hardy spaces $H_{X,L}(\mathbb{X})$. Moreover,
all these results have a wide range of generality and, particularly, even when they are applied to weighted Lebesgue spaces, mixed-norm Lebesgue spaces, Orlicz spaces, variable Lebesgue spaces and Euclidean spaces setting, these results are also new.}
\end{minipage}
\end{center}

\tableofcontents


\section{Introduction}\label{s1}
\subsection{Background and Motivation}
\hskip\parindent
Consider the fractional Laplace operator $(-\Delta)^{\gamma/2}$ with $\gamma>0$ on the Euclidean space $\rn$ and the
Schr\"{o}dinger equation
$$\lf\{ \begin{array}{ll}
i\partial_\tau u+(-\Delta)^{\gamma/2}u=0,\\
u(x,0)=f(x).
\end{array}\r.$$
Its any solution $u$ satisfies
$$\|u(\cdot,\tau)\|_{L^p(\rn)}=\lf\|e^{i\tau(-\Delta)^{\gamma/2}}f\r\|_{L^p(\rn)}\lesssim
(1+|\tau|)^{n|\frac1{2}-\frac1{p}|}\|f\|_{W^{s,p}(\rn)},$$
where $\tau\in \mathbb{R}$, $0<\gamma\neq1$, $p\in(1,\infty)$ and $\beta=\gamma n|\frac1{2}-\frac1{p}|$.
It is well-known that the Schr\"{o}dinger flow $e^{i\tau(-\Delta)^{\gamma/2}}$ with $\gamma>0$ is a group of isometries on
$L^2(\rn)$ but is unbounded on every other $L^p$ space with $p\neq2$. For $p\neq2$, it was shown
that for $\beta=\gamma n|\frac1{2}-\frac1{p}|$, the operator
$(I-\Delta)^{-\beta/2}e^{i\tau(-\Delta)^{\gamma/2}}$ is bounded on $H^p(\rn)$ with $p\in(0,\infty)$ (see, for instance, \cite{bdd24}),
where $H^p(\rn)$ is the classical Hardy space and $H^p(\rn)=L^p(\rn)$ if $p\in(1,\infty)$ (see, for instance, \cite{fs72}).

Let $L$ be a non-negative self-adjoint operator
satisfying the Davies-Gaffney estimate on spaces of homogeneous type space.
Since the flows $e^{i\tau L^{\gamma/2}}$ are well defined for arbitrary nonnegative self-adjoint operators via spectral calculus, Bui et al. \cite{bdd24} obtained the sharp estimate for $e^{i\tau L^{\gamma/2}}$
on Hardy space $H^p_{L}(\mathbb{X})$ associated with $L$,
and showed that for any $\tau\in \mathbb{R}$ and $p\in(0,1]$,
\begin{align}\label{e1.1}
\lf\|(I+L)^{-\beta/2}e^{i\tau L^{\gamma/2}}f\r\|_{H^p_{L}(\mathbb{X})}\lesssim\lf(1+|\tau|\r)^{n|\frac1{2}-\frac1{p}|}\|f\|_{H^p_{L}(\mathbb{X})},
\end{align}
where $0<\gamma\neq1$ and $\beta\in[\gamma n|\frac1{2}-\frac1{p}|,\infty)$. In the special case of the Hermite operator $L$ on $\rn$ with $n\geq2$,
the sharp estimate \eqref{e1.1} on Lebesgue space $L^p(\rn)$ with $p\in(1,\infty)$ and Hardy space $H^p_{L}(\rn)$ with $p\in(0,1]$ associated with $L$ was obtained in \cite{bdhh23}.
For more information about Schr\"{o}dinger flows beyond the Laplacian case, we refer the
reader to \cite{bddm19,bdn20,cdlsy23,cdly20,cdly23}.

Moreover, the study for the real-variable theory of (weak) Hardy type spaces
associated with the ball quasi-Banach function space $X$ has inspired great interests in recent years
(see, for instance, \cite{dlyyz23,lyyy25,shyy17,wyy20,wyyz21,yhyy22,yhyy23,zyyw21}).
In particular, when $L$ is a non-negative self-adjoint operator on $L^2(\rn)$
with its heat kernels satisfying the Gaussian upper bound estimate, Lin et al. \cite{lyyy24}
introduced the Hardy space $H_{X,L}(\rn)$ associated with $L$ and $X$, and then established the radial and the
non-tangential maximal function characterizations of $H_{X,L}(\rn)$.
Meanwhile, Lin et al. \cite{lyyy26} introduced the Hardy space $H_{X,L}(\mathbb{X})$,
associated with $X$ and an operator $L$ whose heat kernel satisfies the Davies-Gaffney estimate, on spaces of homogeneous type
$\mathbb{X}$, and then several characterizations of $H_{X,L}(\mathbb{X})$,
in terms of the atomic and molecular characterizations, the
boundedness of spectral multiplies, and the boundedness estimate of Schr\"{o}dinger groups
when the operator $L$ further satisfies the Gaussian upper bound estimate, are obtained.
Compared with quasi Banach function spaces, ball quasi Banach function spaces contain more function spaces and
hence are more general (see, for instance, \cite{lyyy24,lyyy26,shyy17}).

Let $(\mathbb{X},d,\mu)$ be a doubling metric measure space, $L$ a non-negative self-adjoint operator on $L^2(\mathbb{X})$
satisfying the Davies-Gaffney estimate, and $X(\mathbb{X})$ a ball quasi-Banach function space on $\mathbb{X}$ satisfying some mild assumptions with $p\in(0,\infty)$ and $s_0\in(0,\min\{p,1\}]$. Motivated by \cite{bbhh22,bdd24,bdhh23,lyyy24,lyyy26},
in this article, we first introduce the weak Hardy space $WH_{X,L}(\mathbb{X})$ via the Lusin area function associated with $L$,
and then establish the atomic and the molecular characterizations of $WH_{X,L}(\mathbb{X})$.
As applications, we establish the boundedness estimate of Schr\"{o}dinger groups for fractional powers of $L$ on $WH_{X,L}(\mathbb{X})$:
 $$\lf\|(I+L)^{-\beta/2}e^{i\tau L^{\gamma/2}}f\r\|_{WH_{X,L}(\mathbb{X})}\leq C\lf(1+|\tau|\r)^{n(\frac{1}{s_0}-\frac{r}{2})}\|f\|_{WH_{X,L}(\mathbb{X})},$$
where $0<\gamma\neq1$, $\beta\in[\gamma n(\frac{1}{s_0}-\frac{1}{2}),\infty)$,
$r\in(0,1]$, $\tau\in \mathbb{R}$, and $C>0$ is a constant.
Moreover, when $(\mathbb{X},d,\mu)$ is an Ahlfors $n$-regular metric measure space, $L$ satisfies the Gaussian upper bound estimate, and the kernels of the spectral multiplier operators $F(L)$ satisfy an appropriate weighted $L^2$ estimate,
we obtain the boundedness estimate of imaginary power operators of $L$ on $WH_{X,L}(\mathbb{X})$:
$$\lf\|L^{i\tau}f\r\|_{WH_{X,L}(\mathbb{X})}\leq C\lf(1+|\tau|\r)^{n(\frac{1}{s_0}-\frac{r}{2})}\|f\|_{WH_{X,L}(\mathbb{X})},$$
where $\alpha>n(\frac{1}{s_0}-\frac{1}{2})$, $r\in(\frac{n/s_0}{\alpha+n/2},1]$, $\tau\in \mathbb{R}$, and $C>0$ is a constant.
Moreover, these results also novelty for strong Hardy spaces $H_{X,L}(\mathbb{X})$.
All these results have a wide range of generality and, particularly, even when they are applied to weighted Lebesgue spaces, mixed-norm Lebesgue spaces, Orlicz spaces, variable Lebesgue spaces and Euclidean spaces setting, these results are also new.

\subsection{Main Results and Organization}
\hskip\parindent
In this subsection, we will state the main
results and  the organization of this paper.
To state the main results of this article, now we describe the concept of spaces of homogeneous type initially introduced by
Coifman and Weiss \cite{cw77}. Let $(\mathbb{X},d,\mu)$ be a metric measure space endowed with a distance $d$ and a non-negative Borel measure $\mu$
satisfy the following doubling condition: there exists a constant $C_1\in(0,\infty)$ such that, for any $x\in\mathbb{X}$ and $r\in(0,\infty)$,
\begin{equation}\label{e1.2}
V(x,2r)\leq C_1V(x,r)<\infty,
\end{equation}
where $V(x,r):=\mu(B(x,r))$ and $B(x,r):=\{y\in\mathbb{X}:d(x,y)<r\}$. Moreover, notice that the doubling property \eqref{e1.2} implies that
the following strong homogeneity property that, for some positive constants $C_2$ and $n$,
\begin{equation}\label{e1.3}
V(x,\lambda r)\leq C_2\lambda^n V(x,r)
\end{equation}
uniformly holds true for any $\lambda\in[1,\infty)$, $x\in\mathbb{X}$ and $r\in(0,\infty)$. Further implies that there exist constants
$C_3\in(0,\infty)$ and $n_0\in[0,n]$ such that, for any $x,y\in\mathbb{X}$ and $r\in(0,\infty)$,
$$V(x,r)\leq C_3\lf[1+\frac{d(x,y)}{r}\r]^{n_0}V(y,r).$$

The symbol $\mathfrak{U}(\mathbb{X})$ denotes the set of all $\mu$-measurable functions on $\mathbb{X}$.
Next, we recall the concept of ball (quasi-)Banach function spaces on $\mathbb{X}$ (see \cite[Definition 2.2 and (2.3)]{shyy17} for the case of $\rn$).
\begin{definition}\label{de1.1}
A quasi-normed linear space $X(\mathbb{X})\subset \mathfrak{U}(\mathbb{X})$, equipped with a {\it quasi-norm}
$\|\cdot\|_{X(\mathbb{X})}$ which makes sense for all functions in $\mathfrak{U}(\mathbb{X})$, is called a {\it ball quasi-Banach function space} (for short, BQBF space) on $\mathbb{X}$ if it satisfies the following conditions:
\begin{enumerate}
\item[\rm{(i)}]for any $f\in \mathfrak{U}(\mathbb{X})$, $\|f\|_{X(\mathbb{X})}=0$ if and only if $f=0$ $\mu$-almost everywhere;
\item[\rm{(ii)}] for any $f,g\in \mathfrak{U}(\mathbb{X})$, $|g|\leq|f|$ $\mu$-almost everywhere implies that
$\|g\|_{X(\mathbb{X})}\leq\|f\|_{X(\mathbb{X})}$;
\item[\rm{(iii)}]for any $\{f_k\}_{k\in\nn}\subset \mathfrak{U}(\mathbb{X})$ and $f\in\mathfrak{U}(\mathbb{X})$,
$0\leq f_k\uparrow f$ $\mu$-almost everywhere as $k\rightarrow\infty$ implies that $\|f_k\|_{X(\mathbb{X})}\uparrow\|f\|_{X(\mathbb{X})}$ as $k\rightarrow\infty$;
\item[\rm{(iv)}] for any ball $B\subset\mathbb{X}$, $\mathbf{1}_{B}\in X(\mathbb{X})$,
where $\mathbf{1}_{B}$ is defined as the \emph{characteristic function} of $B$.
\end{enumerate}
A quasi-normed linear space $X(\mathbb{X})\subset \mathfrak{U}(\mathbb{X})$ is called a {\it ball Banach function space} (for short, BBF space) on $\mathbb{X}$ if it satisfies (i)-(iv) and
\begin{enumerate}
\item[\rm{(v)}] for any ball $B$ of $\mathbb{X}$, there exists a positive constant $C$, depending only on $B$, such that, for any $f\in X(\mathbb{X})$, $$\int_B|f(x)|d\mu(x)\leq C\|f\|_{X(\mathbb{X})}.$$
\end{enumerate}
\end{definition}

The {\it associate space} $X'(\mathbb{X})$ of a BBF space $X(\mathbb{X})$ is defined by setting
$$X'(\mathbb{X}):=\{f\in\mathfrak{U}(\mathbb{X}):\|f\|_{X'(\mathbb{X})}<\infty\},$$
where, for any $f\in\mathfrak{U}(\mathbb{X})$,
$$\|f\|_{X'(\mathbb{X})}:=\sup\lf\{\|fg\|_{L^1(\mathbb{X})}:g\in X(\mathbb{X}),\|g\|_{X(\mathbb{X})}=1\r\}$$
(see, for instance, \cite[Definition 2.1]{bs88}).

The following concept of the convexification of a BQBF space can be found in
\cite[Definition 2.6]{shyy17}.
\begin{definition}
 Let $p\in(0,\infty)$ and $X(\mathbb{X})$ be a BQBF space. The {\it $p$-convexification}
$X^p(\mathbb{X})$ of $X(\mathbb{X})$ is defined by setting
$$X^p(\mathbb{X}):=\lf\{f\in\mathfrak{U}(\mathbb{X}): \|f\|_{X^p(\mathbb{X})}:=
\lf\||f|^p\r\|^{\frac1{p}}_{X(\mathbb{X})}<\infty \r\}.$$
\end{definition}
Recall that the {\it Hardy--Littlewood maximal operator} $\mathcal{M}$ on $\mathbb{X}$ is defined by setting, for any
$f\in\mathfrak{U}(\mathbb{X})$ and $x\in\mathbb{X}$,
$$\mathcal{M}(f)(x):=\sup_{B\ni x}\frac1{\mu(B)}\int_B|f(y)|\,d\mu(y),$$
where the supremum is taken over all the balls $B$ containing $x$.

Moreover, we also need the following two key assumptions on BQBF spaces.
\begin{assumption}\label{as1.1}
Let $X(\mathbb{X})$ be a BQBF space. Assume that there exists a positive constant $p$ such
that, for any given $t\in(0,p)$ and $u\in(1,\infty)$. Then there exists a positive constant $C$ such that, for any
$\{f_j\}_{j\in\nn}\subset\mathfrak{U}(\mathbb{X})$,
$$\lf\|\lf\{\sum_{j\in\nn}\lf[\mathcal{M}\lf(f_j\r)\r]^u\r\}^{\frac1{u}}\r\|_{X^{\frac1{t}}(\mathbb{X})}
\leq C\lf\|\lf(\sum_{j\in\nn}\lf|f_j\r|^u\r)^{\frac1{u}}\r\|_{X^{\frac1{t}}(\mathbb{X})}.$$
\end{assumption}

\begin{assumption}\label{as1.2}
Let $X(\mathbb{X})$ be a BQBF space. Assume that there exist constants $s_0\in(0,\infty)$ and $q_0\in(s_0,\infty)$ such
that $X^{\frac1{s_0}}(\mathbb{X})$ is a BBF space and the Hardy--Littlewood maximal operator $\mathcal{M}$ is
bounded on the $\frac1{(q_0/s_0)'}$-convexification of the associate space $(X^{\frac1{s_0}})'(\mathbb{X})$,
where $\frac1{(q_0/s_0)'}+\frac1{q_0/s_0}=1$.
\end{assumption}

Let $L$ be a non-negative self-adjoint operator on $L^2(\mathbb{X})$. In this article, we always make the
following assumption on the operator $L$.
\begin{assumption}\label{as1.3}
The semigroup $\{e^{-tL}\}_{t\in(0,\infty)}$, generated by $-L$, satisfies the {\it Davies-Gaffney
estimate}, that is, there exist positive constants $C$ and $c$ such that, for any closed sets $E,F\subset\mathbb{X}$
and any $f\in L^2(\mathbb{X})$ with  $\supp(f)\subset E$,
\begin{align}\label{e1.4}
\lf\|e^{-tL}(f)\r\|_{L^2(F)}\leq C\exp\lf\{-\frac{[\mathrm{dist}(E,F)]^2}{ct}\r\}\|f\|_{L^2(E)},
\end{align}
where $\mathrm{dist}(E,F):=\inf\{d(x,y):\ x\in E,\,y\in F\}$.
\end{assumption}

Let the weak Hardy spaces $WH_{X,L}(\mathbb{X})$, $WH^M_{X,L,\mathrm{atom}}(\mathbb{X})$, and $WH^{M,\varepsilon}_{X,L,\mathrm{mol}}(\mathbb{X})$ associated with $L$ and $X$
be as in Definitions \ref{de2.2}, \ref{de2.4}, and \ref{de2.3}, respectively. Then we have the following atomic and molecular characterizations of $WH_{X,L}(\mathbb{X})$.
\begin{theorem}\label{th1.1}
Let $X(\mathbb{X})$ be a ${\rm{BQBF}}$ space satisfying both Assumptions \ref{as1.1} and \ref{as1.2} for some
$p\in(0,\infty)$, $s_0\in(0,\min\{p,1\}]$, and $q_0\in(s_0,2]$.
Assume that $L$ is a non-negative self-adjoint operator on $L^2(\mathbb{X})$
satisfying the Davies-Gaffney estimate \eqref{e1.4}. Let $M\in(\frac{n}{2}[\frac{1}{s_0}-\frac{1}{2}],\infty)\cap\nn$
and $\varepsilon\in(\frac{n}{s_0},\infty)$.
Then we have $$WH_{X,L}(\mathbb{X})=WH^M_{X,L,\mathrm{atom}}(\mathbb{X})=WH^{M,\varepsilon}_{X,L,\mathrm{mol}}(\mathbb{X})$$
with equivalent quasi-norms.
\end{theorem}

\begin{remark}
We should  mention that, in the last few days, we learned that
X. Lin, D. Yang, S. Yang and W. Yuan \cite{lyyy25x} also have
independently obtained the same result with Theorem \ref{th1.1} when our paper was ready for submission. Therefore, we omit its proof here. For the concrete proof of Theorem \ref{th1.1}, we refer the reader to see \cite[Theorems 3.4 and 3.5]{lyyy25x}.
\end{remark}


\begin{theorem}\label{th1.2}
Let $X(\mathbb{X})$ be a ${\rm{BQBF}}$ space satisfying both Assumptions \ref{as1.1} and \ref{as1.2} for some
$p\in(0,\infty)$, $s_0\in(0,\min\{p,1\}]$, and $q_0\in(s_0,2]$.
Assume that $L$ is a non-negative self-adjoint operator on $L^2(\mathbb{X})$
satisfying the Davies-Gaffney estimate \eqref{e1.4}. Let $0<\gamma\neq1$,
$\beta\in[\gamma n(\frac{1}{s_0}-\frac{1}{2}),\infty)$, and $r\in(0,1]$. Then there exists a positive
constant $C$ such that, for any $\tau\in \mathbb{R}$ and $f\in WH_{X,L}(\mathbb{X})$,
\begin{align}\label{e1.5}
\lf\|(I+L)^{-\beta/2}e^{i\tau L^{\gamma/2}}f\r\|_{WH_{X,L}(\mathbb{X})}\leq C\lf(1+|\tau|\r)^{n(\frac{1}{s_0}-\frac{r}{2})}\|f\|_{WH_{X,L}(\mathbb{X})}.
\end{align}
\end{theorem}

Theorem \ref{th1.2} is proved by using the atomic characterization of $WH_{X,L}(\mathbb{X})$ obtained by
Theorem \ref{th1.1}, the functional calculi associated with $L$,
and some more subtle norm estimates obtained by \eqref{e4.6} and \eqref{e4.7} below.

By the fact that the Hardy space $H_{X,L}(\mathbb{X})$ (see Definition \ref{de2.2} for its definition) continuously embeds into
$WH_{X,L}(\mathbb{X})$, namely, if $f\in H_{X,L}(\mathbb{X})$, then $f\in WH_{X,L}(\mathbb{X})$ and
$\|f\|_{WH_{X,L}(\mathbb{X})}\leq\|f\|_{H_{X,L}(\mathbb{X})}$ (see, for instance, \cite[Proposition 2.16]{syy22b}).
Then, we have the following conclusion. Since their
proofs are similar to that of Theorem \ref{th1.2} and \cite[Theorem 4.16]{lyyy24}, we omit the details here.
\begin{theorem}\label{co1.1}
Let $X(\mathbb{X})$ be a ${\rm{BQBF}}$ space satisfying both Assumptions \ref{as1.1} and \ref{as1.2} for some
$p\in(0,\infty)$, $s_0\in(0,\min\{p,1\}]$, and $q_0\in(s_0,2]$.
Assume that $L$ is a non-negative self-adjoint operator on $L^2(\mathbb{X})$
satisfying the Davies-Gaffney estimate \eqref{e1.4}. Let $0<\gamma\neq1$
and $\beta\in[\gamma n(\frac{1}{s_0}-\frac{1}{2}),\infty)$.
Then there exists a positive constant $C$ such that, for any $\tau\in \mathbb{R}$ and $f\in H_{X,L}(\mathbb{X})$,
\begin{align*}
\lf\|(I+L)^{-\beta/2}e^{i\tau L^{\gamma/2}}f\r\|_{H_{X,L}(\mathbb{X})}\leq C\lf(1+|\tau|\r)^{n(\frac{1}{s_0}-\frac{1}{2})}\|f\|_{H_{X,L}(\mathbb{X})}.
\end{align*}
\end{theorem}

Applying Theorem \ref{co1.1} with $X(\mathbb{X}):=L^r(\mathbb{X})$, we obtain the following conclusion.
\begin{corollary}\label{co1.2}
Let $L$ be a non-negative self-adjoint operator on $L^2(\mathbb{X})$
satisfying the Davies-Gaffney estimate \eqref{e1.4}, $r\in(0,2)$, $0<\gamma\neq1$, and $\beta\in[\gamma n(\frac{1}{\min\{1,r\}}-\frac{1}{2}),\infty)$. Then there exists a positive constant $C$ such that, for any $\tau\in \mathbb{R}$ and $f\in H^r_{L}(\mathbb{X})$,
\begin{align*}
\lf\|(I+L)^{-\beta/2}e^{i\tau L^{\gamma/2}}f\r\|_{H^r_{L}(\mathbb{X})}\leq C\lf(1+|\tau|\r)^{n(\frac{1}{\min\{1,r\}}-\frac{1}{2})}\|f\|_{H^r_{L}(\mathbb{X})}.
\end{align*}
\end{corollary}

\begin{remark}\label{re1.1}
\begin{enumerate}
\item[\rm{(i)}] Since the operator $L$ satisfies the Davies-Gaffney estimate, which is weaker than the Gaussian upper bound estimate, Theorem \ref{co1.1} and Corollary \ref{co1.2} improve the known results in \cite[Theorems 4.16 and 4.17]{lyyy26} even when $\gamma\equiv 2$.
\item[\rm{(ii)}] We need to point out that \cite[Theorem 1.2]{bdd24} shows that Corollary \ref{co1.2} with $r\in(0,1]$ and
$\beta\in[\gamma n(\frac{1}{r}-\frac{1}{2}),\infty)$ also holds true.
\end{enumerate}
\end{remark}

Another important application of Theorem \ref{th1.2} involves the Riesz means associated with $L$, which are defined through the following operator:
$$I_{s,t}(L):=st^{-s}\int^t_0(t-\lambda)^{-s-1}e^{-i\lambda L^{\gamma/2}}d\lambda,$$
where $s,t>0$, while $I_{s,t}(L)=\bar{I}_{s,-t}(L)$ for $t<0$ (see, for instance, \cite[p.264]{bdd24}).
Applying Theorem \ref{th1.2} and standard arguments, we have the following conclusion, with the proof details omitted.
\begin{corollary}\label{co1.3}
Let $X(\mathbb{X})$ be a ${\rm{BQBF}}$ space satisfying both Assumptions \ref{as1.1} and \ref{as1.2} for some
$p\in(0,\infty)$, $s_0\in(0,\min\{p,1\}]$, and $q_0\in(s_0,2]$.
Assume that $L$ is a non-negative self-adjoint operator on $L^2(\mathbb{X})$
satisfying the Davies-Gaffney estimate \eqref{e1.4}. Let $0<\gamma\neq1$
and $s\in[\gamma n(\frac{1}{s_0}-\frac{1}{2}),\infty)$.
Then there exists a positive constant $C$ such that, for any $t>0$ and $f\in WH_{X,L}(\mathbb{X})$,
\begin{align*}
\lf\|I_{s,t}(L)f\r\|_{WH_{X,L}(\mathbb{X})}\leq C\|f\|_{WH_{X,L}(\mathbb{X})}.
\end{align*}
\end{corollary}

As applications, we obtain the boundedness estimate of imaginary power operators of $L$ on $WH_{X,L}(\mathbb{X})$.
To describe the details, we begin with recalling the definitions of imaginary power operators.
We assume that $(\mathbb{X},d,\mu)$ is an Ahlfors $n$-regular metric measure space, which means that there exist constants
$n,C_1,C_2$ such that, for any $x\in\mathbb{X}$ and $r>0$,
\begin{align}\label{e1.6}
C_1r^n\leq\mu(B(x,r))\leq C_2r^n.
\end{align}

Let $L$ be a non-negative self-adjoint operator on $L^2(\mathbb{X})$. In this article, we need the
following assumption on $L$.
\begin{assumption}\label{as1.4}
The kernels of the semigroup $\{e^{-tL}\}_{t\in(0,\infty)}$, denoted by $\{K_t\}_{t\in(0,\infty)}$, are measurable functions on
$\mathbb{X}\times\mathbb{X}$ and satisfy the {\it Gaussian upper bound estimate},
that is, there exist positive constants $C$ and $c$ such that, for any $t\in(0,\infty)$ and $x,y\in\mathbb{X}$,
\begin{align}\label{e1.7}
\lf|K_t(x,y)\r|\leq \frac{C}{V(x,\sqrt{t})}\mathrm{exp}\lf\{-\frac{[d(x,y)]^2}{ct}\r\}.
\end{align}
\end{assumption}

By the spectral theory, $L$ admits a spectral resolution
$$L=\int^\infty_0\lambda dE(\lambda),$$
where $\{E(\lambda):\lambda\geq0\}$ is the spectral resolution of $L$. If $F$ is a bounded Borel measurable function on
$[0,\infty)$, then the operator
$$F(L)=\int^\infty_0F(\lambda) dE(\lambda)$$
is bounded on $L^2(\mathbb{X})$. In what follows, we use $K_{F(L)}(x,y)$ to denote the kernel of $F(L)$.
Before considering the imaginary power operator $L^{i\tau}$, we also need the following
additional condition regarding the weighted $L^2$ bound for the kernels of the spectral multiplier operators $F(L)$.
\begin{assumption}\label{as1.5}
There exist a positive constant $C$ and $q\in[2,\infty]$ such that, for any even positive integer
$\alpha\in2\nn$ and any $R\in(0,\infty)$,
\begin{align}\label{e1.8}
\int_G\lf|K_{F(L)}(x,y)\r|^2\lf[1+Rd(x,y)\r]^{2\alpha}d\mu(x)\leq CR^n\lf\|\delta_{R^2}F\r\|^2_{W^q_\alpha(\mathbb{R})}
\end{align}
holds uniformly in $y\in\mathbb{X}$ for any bounded Borel functions $F:[0,\infty)\rightarrow\mathbb{C}$
supported in $[\frac{R^2}{4},R^2]$, where
$$\delta_{R^2}F(\lambda):=F(R^2\lambda)$$
and
$$\|F\|_{W^q_\alpha(\mathbb{R})}:=\lf\|\lf(I-d^2/d\lambda^2\r)^{\alpha/2}F\r\|_{L^q(\mathbb{R})}.$$
\end{assumption}
\begin{theorem}\label{th1.3}
Let $(\mathbb{X},d,\mu)$ be a metric measure space satisfying the Ahlfors $n$-regular condition \eqref{e1.6}, and $L$ be a non-negative self-adjoint operator on $L^2(\mathbb{X})$
satisfying the Gaussian upper estimate \eqref{e1.7} and \eqref{e1.8}, $X(\mathbb{X})$ be a ${\rm{BQBF}}$ space satisfying both Assumptions \ref{as1.1} and \ref{as1.2} for some
$p\in(0,\infty)$, $s_0\in(0,\min\{p,1\}]$, and $q_0\in(s_0,2]$.
Assume that $\alpha>n(\frac{1}{s_0}-\frac{1}{2})$ and $r\in(\frac{n/s_0}{\alpha+n/2},1]$. Then there exists a positive constant $C$ such that, for any $\tau\in \mathbb{R}$ and $f\in WH_{X,L}(\mathbb{X})$,
\begin{align}\label{e1.9}
\lf\|L^{i\tau}f\r\|_{WH_{X,L}(\mathbb{X})}\leq C\lf(1+|\tau|\r)^{n(\frac{1}{s_0}-\frac{r}{2})}\|f\|_{WH_{X,L}(\mathbb{X})}.
\end{align}
\end{theorem}

 We show Theorem \ref{th1.3} via borrowing some ideas from the proof  of \cite[Theorem 1.1]{bbhh22},
the atomic characterization of $WH_{X,L}(\mathbb{X})$ obtained by Theorem \ref{th1.1},
and using some more subtle norm estimates in \eqref{e5.6} and \eqref{e5.7} below.

By the fact that the Hardy space $H_{X,L}(\mathbb{X})$ (see Definition \ref{de2.2} for its definition) continuously embeds into
$WH_{X,L}(\mathbb{X})$ (see, for instance, \cite[Proposition 2.16]{syy22b}).
Then, we have the following conclusion, whose
proof is similar to that of Theorem \ref{th1.3} and \cite[Theorem 4.16]{lyyy24}. To control the length of this paper, we leave the details to the interested readers.
\begin{theorem}\label{co1.4}
Let $(\mathbb{X},d,\mu)$ be a metric measure space satisfying the Ahlfors $n$-regular condition \eqref{e1.6}.
Let $L$ be a non-negative self-adjoint operator on $L^2(\mathbb{X})$
satisfying the Gaussian upper estimate \eqref{e1.7} and \eqref{e1.8}.
Assume that $X(\mathbb{X})$ is a ${\rm{BQBF}}$ space satisfying both Assumptions \ref{as1.1} and \ref{as1.2} for some
$p\in(0,\infty)$, $s_0\in(0,\min\{p,1\}]$, and $q_0\in(s_0,2]$.
Let $\alpha>\frac{n}{2}$. Then there exists a positive constant $C$ such that, for any $\tau\in \mathbb{R}$ and $f\in H_{X,L}(\mathbb{X})$,
\begin{align*}
\lf\|L^{i\tau}f\r\|_{H_{X,L}(\mathbb{X})}\leq C\lf(1+|\tau|\r)^{n(\frac{1}{s_0}-\frac{1}{2})}\|f\|_{H_{X,L}(\mathbb{X})}.
\end{align*}
\end{theorem}

Using Theorem \ref{co1.4} with $X(\mathbb{X}):=L^r(\mathbb{X})$, we obtain the following corollary.
The proof of this corollary is similar to that Corollary \ref{co1.2}, the details being omitted here.
\begin{corollary}\label{co1.5}
Let $(\mathbb{X},d,\mu)$ be a metric measure space satisfying the Ahlfors $n$-regular condition \eqref{e1.6}.
Let $L$ be a non-negative self-adjoint operator on $L^2(\mathbb{X})$
satisfying the Gaussian upper estimate \eqref{e1.7} and \eqref{e1.8}.
Let $\alpha>\frac{n}{2}$ and $r\in(0,2)$. Then there exists a positive constant $C$ such that, for any $\tau\in \mathbb{R}$ and $f\in H^r_{L}(\mathbb{X})$,
\begin{align*}
\lf\|L^{i\tau}f\r\|_{H^r_{L}(\mathbb{X})}\leq C\lf(1+|\tau|\r)^{n(\frac{1}{\min\{1,r\}}-\frac{1}{2})}\|f\|_{H^r_{L}(\mathbb{X})}.
\end{align*}
\end{corollary}

\begin{remark}\label{re1.2}
\begin{enumerate}
\item[\rm{(i)}]
Let $r\in(0,\infty)$. In this case, Corollary \ref{co1.5} was obtained in \cite[Theorem 1.1]{bbhh22}.
\item[\rm{(ii)}]
We should point out that, even when it comes to Euclidean spaces setting ($\mathbb{X}:=\rn$),
Theorems \ref{th1.2}, \ref{co1.1},  \ref{th1.3} and  \ref{co1.4} are also new.
\item[\rm{(iii)}] As is known to all, there are many non-negative self-adjoint operators and their heat satisfies the Gaussian upper estimate \eqref{e1.7} and \eqref{e1.8}.
For example, when $\mathbb{X}:=\rn$, the Hermit operator
$$\mathcal{H}:=-\Delta+|x|^2$$
is also a non-negative self-adjoint operators on $L^2(\rn)$ and its heat satisfies the Gaussian upper estimate \eqref{e1.7} and \eqref{e1.8}.
Therefore, our main results pave the way for further studies such
as establishing the boundedness estimate of Schr\"{o}dinger groups for fractional powers of $L:=\mathcal{H}$
 and the imaginary power operators of $L:=\mathcal{H}$ on $H_{X,L}(\rn)$ or $WH_{X,L}(\rn)$.
For more progress about the Hermit operator $\mathcal{H}$, we refer the reader to see \cite{bbhh22,bdhh23}. To limit the
length of this paper, we do not pursue these problems here.

\end{enumerate}
\end{remark}

The organization of this article is as follows. In Section \ref{s2}, we present some notions and some basic properties
of the ball quasi-Banach function spaces and the weak Hardy spaces used in this paper.
The proofs of Theorem \ref{th1.2} and Corollary \ref{co1.2} are presented in Section \ref{s4}.
In Section \ref{s5}, we will give some preliminary results related
to the discrete square function characterization for $WH_{X,L}(\mathbb{X})$ space,
in the remainder of this section, we show Theorem \ref{th1.3}.
In Section \ref{s6}, we apply the results obtained in Sections \ref{s4} and \ref{s5} to some specific spaces including weighted Lebesgue spaces, Orlicz spaces, mixed-norm Lebesgue spaces, and variable Lebesgue spaces associated with operators.

\subsection{Notation}
\hskip\parindent
At the end of this section, we make some conventions on notation. Let $\nn:=\{1,\, 2,\,\ldots\}$ and $\zz_+:=\{0\}\cup\nn$.
Throughout the whole article, we denote by $C$ a \emph{positive
constant} which is independent of the main parameters, but it may
vary from line to line. The \emph{symbol} $f\ls g$ means that $f\le Cg$. If $f\ls
g$ and $g\ls f$, then we write $f\sim g$.
For each ball $B:=B(x_B,r_B)$ of $\mathbb{X}$, with \emph{center} $x_B\in\mathbb{X}$ and \emph{radius}
$r_B\in(0,\fz)$, and $\delta\in(0,\fz)$, let $\delta B:=B(x_B,\delta r_B)$.
For any $k\in\nn$ and any ball $B\subset\mathbb{X}$, let $U_k(B):=(2^kB)\setminus(2^{k-1}B)$ and $U_0(B):=B$.
The symbol $\mathfrak{U}(\mathbb{X})$ denotes the set of all $\mu$-measurable functions on $\mathbb{X}$.
For any measurable subset $E$ of $\mathbb{X}$, we denote the \emph{set} $\mathbb{X}\setminus E$ by $E^\complement$
and its \emph{characteristic function} by $\mathbf{1}_{E}$.
Finally, for any given $q\in[1,\fz]$, we denote by $q'$
its \emph{conjugate exponent}, namely, $1/q + 1/q'= 1$.



\section{Preliminaries \label{s2}}
\hskip\parindent
In this section, we describe some basic properties of the ball quasi-Banach function spaces and  the weak Hardy spaces.

\subsection{Basic Properties of (Weak) Ball Quasi-Banach Function Spaces}
\hskip\parindent
The following two lemmas are just \cite[Proposition 4.8(i)]{syy22b} and \cite[Proposition 2.14]{syy22a}, respectively.
\begin{lemma}\label{le2.1}
Let $X(\mathbb{X})$ be a ${\rm{BQBF}}$ space satisfying Assumption \ref{as1.1} for some $p\in(0,\infty)$. Then, for any given
$t\in(0,\infty)$ and $s\in(\max\{1,t/p\},\infty)$, there exists a positive constant $C$ such that, for any $\tau\in[1,\infty)$,
any $\{\lambda_j\}_{j\in\nn}\in[0,\infty)$, and any sequence $\{B_j\}_{j\in\nn}$ of balls
$$\lf\|\sum_{j\in\nn}\lambda_j\mathbf{1}_{\tau B_j}\r\|_{X^{\frac1{t}}(\mathbb{X})}
\leq C\tau^{sn}\lf\|\sum_{j\in\nn}\lambda_j\mathbf{1}_{B_j}\r\|_{X^{\frac1{t}}(\mathbb{X})},$$
where $n$ is as in \eqref{e1.3}.
\end{lemma}

\begin{lemma}\label{le2.2}
Let $X(\mathbb{X})$ be a ${\rm{BQBF}}$ space satisfying Assumption \ref{as1.2} for some $s_0\in(0,\infty)$
and $q_0\in(s_0,\infty)$. Assume that $q\in[q_0,\infty)$. Then
there exists a positive constant $C$ such that, for any
$\{\lambda_j\}_{j\in\nn}\in(0,\infty)$ and $\{a_j\}_{j\in\nn}\in L^q(\mathbb{X})$ satisfying both
$\|a_j\|_{L^q(\mathbb{X})}\leq\lambda_j[\mu(B_j)]^{\frac1{q}}$ and $\supp(a_j)\subset \theta B_j$ for any $j\in\nn$,
$$\lf\|\sum_{j\in\nn}\lf|a_j\r|^{s_0}\r\|_{X^{\frac1{s_0}}(\mathbb{X})}
\leq C\theta^{(1-\frac{s_0}{q})n}\lf\|\sum_{j\in\nn}\lf|\lambda_j\r|^{s_0}
\mathbf{1}_{B_j}\r\|_{X^{\frac1{s_0}}(\mathbb{X})},$$
where $n$ is as in \eqref{e1.3}.
\end{lemma}

\begin{remark}\label{re2.1}
Let $X(\mathbb{X})$ be a ${\rm{BQBF}}$ space satisfying Assumption \ref{as1.1} for some $p\in(0,\infty)$,
$\{B_j\}_{j\in\nn}$ a sequence of balls of $\mathbb{X}$ satisfying that there exist $c\in(0,1]$ and
$M_0\in\nn$ such that $\sum_{j\in\nn}\mathbf{1}_{cB_j}\leq M_0$, and $s_0\in(0,\infty)$. Then, by Lemma \ref{le2.1}
and the assumption that $X^{\frac{1}{s_0}}(\mathbb{X})$ is a BBF space, we know that
\begin{align*}
\Bigg\|\Bigg(\sum_{j\in\nn}\mathbf{1}_{B_j}\Bigg)^{1/\min\{s_0,1\}}\Bigg\|_{X(\mathbb{X})}
&=\bigg\|\sum_{j\in\nn}\mathbf{1}_{B_j}\bigg\|^{1/\min\{s_0,1\}}_{X^{1/\min\{s_0,1\}}(\mathbb{X})}
\lesssim\bigg\|\sum_{j\in\nn}\mathbf{1}_{cB_j}\bigg\|^{1/\min\{s_0,1\}}_{X^{1/\min\{s_0,1\}}(\mathbb{X})}\\
&\sim\Bigg\|\Bigg(\sum_{j\in\nn}\mathbf{1}_{cB_j}\Bigg)^{1/\min\{s_0,1\}}\Bigg\|_{X(\mathbb{X})}
\lesssim\Bigg\|\sum_{j\in\nn}\mathbf{1}_{cB_j}\Bigg\|_{X(\mathbb{X})}\\
&\lesssim\Bigg\|\sum_{j\in\nn}\mathbf{1}_{B_j}\Bigg\|_{X(\mathbb{X})},
\end{align*}
where the equivalent positive constants are independent of $\{B_j\}_{j\in\nn}$.
\end{remark}

Next, we introduce the weak BQBF space on $\mathbb{X}$ (see, for instance, \cite[Definition 2.15]{syy22b}).
\begin{definition}\label{de2.1}
Let $X(\mathbb{X})$ be a ${\rm{BQBF}}$ space. Then the {\it weak ball quasi-Banach function
space} (for short, weak ${\rm{BQBF}}$ space) $WX(\mathbb{X})$ associated with $\mathbb{X}$ is defined to be
the set of all the $\mu$-measurable functions $f$ on $\mathbb{X}$ such that
$$\|f\|_{WX(\mathbb{X})}:=\sup_{\lambda\in(0,\infty)}\lf\{\lambda
\lf\|\mathbf{1}_{\{x\in\mathbb{X}: \ |f(x)|>\lambda\}}\r\|_{X(\mathbb{X})}\r\}<\infty.$$
\end{definition}

By \cite[Proposition 2.16]{syy22b}, we know that the weak BQBF space is also a BQBF space and, moreover,
the BQBF space $X(\mathbb{X})$ continuously embeds into
$WX(\mathbb{X})$, namely, if $f\in X(\mathbb{X})$, then $f\in WX(\mathbb{X})$ and
$\|f\|_{WX(\mathbb{X})}\leq\|f\|_{X(\mathbb{X})}$. Furthermore, let $p\in(0,\infty)$, $W(X^p)(\mathbb{X})$ the weak BQBF
space associated with $X^p(\mathbb{X})$, and $(WX)^p(\mathbb{X})$ the $p$-convexification of $WX(\mathbb{X})$. Notice that, for any $f\in\mathfrak{U}(\mathbb{X})$,
\begin{align*}
\|f\|_{W(X^p)(\mathbb{X})}
:&=\sup_{\lambda\in(0,\infty)}\lf\{\lambda
\lf\|\mathbf{1}_{\{x\in\mathbb{X}: \ |f(x)|>\lambda\}}\r\|_{X^p(\mathbb{X})}\r\}\\
&=\sup_{\lambda\in(0,\infty)}\lf\{\lambda^{\frac1{p}}
\lf\|\mathbf{1}_{\{x\in\mathbb{X}: \ |f(x)|>\lambda\}}\r\|^{\frac1{p}}_{X(\mathbb{X})}\r\}\\
&=:\|f\|_{(WX)^p(\mathbb{X})}.
\end{align*}
Thus, we obtain $W(X^p)(\mathbb{X})=(WX)^p(\mathbb{X})$ with equivalent quasi-norms. In what follows, we use $WX^p(\mathbb{X})$
to denote either $W(X^p)(\mathbb{X})$ or $(WX)^p(\mathbb{X})$ according to our convenience.

To prove Theorem \ref{th1.3}, we also need the following weak-type Fefferman-Stein vector-valued
maximal inequality, which is just \cite[Theorem 4.4]{syy22b}.
\begin{lemma}\label{le2.3}
Let $X(\mathbb{X})$ be a ${\rm{BQBF}}$ space satisfying Assumption \ref{as1.1} for some $p\in(0,\infty)$.
Then, for any given $t\in(0,p)$ and $u\in(1,\infty)$, there exists a positive constant $C$ such that, for any
$\{f_j\}_{j\in\nn}\subset\mathfrak{U}(\mathbb{X})$,
$$\lf\|\lf\{\sum_{j\in\nn}\lf[\mathcal{M}\lf(f_j\r)\r]^u\r\}^{\frac1{u}}\r\|_{WX^{\frac1{t}}(\mathbb{X})}
\leq C\lf\|\lf(\sum_{j\in\nn}\lf|f_j\r|^u\r)^{\frac1{u}}\r\|_{WX^{\frac1{t}}(\mathbb{X})}.$$
\end{lemma}

\subsection{(Weak) Hardy Spaces Associated BQBF Spaces and Non-negative Self-adjoint Operators}
\hskip\parindent
In this subsection, we will recall the definition of the Hardy space $H_{X,L}(\mathbb{X})$, and introduce the weak Hardy space $WH_{X,L}(\mathbb{X})$ associated with the BQBF space $X$  and the non-negative self-adjoint operator $L$.
\begin{definition}\label{de2.2}
Let $WX(\mathbb{X})$ be a weak ${\rm{BQBF}}$ space and $L$ a non-negative self-adjoint operator on $L^2(\mathbb{X})$ satisfying the Davies-Gaffney estimate \eqref{e1.4}. For any $f\in L^2(\mathbb{X})$ and $x\in\mathbb{X}$, the {\it Lusin area function}
$S_{L}(f)$ associated with $L$ is defined by setting
\begin{equation*}
S_{L}(f)(x):=\lf[\int^\infty_0\int_{B(x,t)}
\lf|t^2 L e^{-t^2 L}(f)(y)\r|^2\frac{d\mu(y)dt}{V(x,t)t}\r]^{\frac1{2}}.
\end{equation*}
Then the {\it Hardy space} $H_{X,L}(\mathbb{X})$, associated with $L$ and $X$,
is defined as the completion of the set
$$\mathbb{H}_{X,L}(\mathbb{X}):=\lf\{f\in L^2(\mathbb{X}):
\ \|f\|_{H_{X,L}(\mathbb{X})}:=\big\|S_{L}(f)\big\|_{X(\mathbb{X})}<\infty\r\}.$$
Moreover, the {\it weak Hardy space} $WH_{X,L}(\mathbb{X})$, associated with $L$ and $X$,
is defined as the completion of the set
$$\mathbb{WH}_{X,L}(\mathbb{X}):=\lf\{f\in L^2(\mathbb{X}):
\ \|f\|_{WH_{X,L}(\mathbb{X})}:=\big\|S_{L}(f)\big\|_{WX(\mathbb{X})}<\infty\r\}.$$
\end{definition}

\begin{remark}\label{re2.2}
Let $\Psi\in C_{\mathrm{c}}^\infty(\mathbb{R})$ be an even function such that $\supp(\Psi)\subset\{\xi:\frac1{4}\leq|\xi|\leq4\}$
and $\Psi=1$ on $\{\xi:\frac1{2}\leq|\xi|\leq2\}$. Define
\begin{equation*}
S_{L,\Psi}(f)(x):=\lf[\int^\infty_0\int_{B(x,t)}
\lf|\Psi\lf(t\sqrt{L}\r)(f)(y)\r|^2\frac{d\mu(y)dt}{V(x,t)t}\r]^{\frac1{2}}.
\end{equation*}
Through careful examination of the proofs in \cite{hlmmy11,jy11a}, it is established that the weak Hardy space
$WH_{X,L}(\mathbb{X})$ admits via $S_{L,\Psi}$ rather than $S_{L}$.
\end{remark}

\subsection{Weak Molecular and Atomic
Hardy Spaces}
\hskip\parindent
Now we will introduce the weak molecular Hardy space $WH^{M,\varepsilon}_{X,L,\mathrm{mol}}(\mathbb{X})$ and the weak atomic
Hardy space $WH^M_{X,L,\mathrm{atom}}(\mathbb{X})$ associated with $L$ and $X$ as follows.
\begin{definition}\label{de2.3}
Let $M\in\nn$, $\varepsilon\in(0,\infty)$, and $X(\mathbb{X})$ be a ${\rm{BQBF}}$ space satisfying Assumption \ref{as1.2}
for some $s_0\in(0,\infty)$ and $q_0\in(s_0,\infty)$. Assume that $L$ is a non-negative self-adjoint operator on
$L^2(\mathbb{X})$ satisfying the Davies-Gaffney estimate \eqref{e1.4}.
\begin{enumerate}
\item[\rm{(i)}] A measurable function $m\in L^2(\mathbb{X})$ is called an \emph{$(X,M,\varepsilon)_{L}$-molecule} associated with the ball $B:=B(x_B,r_B)$ of $\mathbb{X}$ for some $x_B\in\mathbb{X}$ and $r_B\in(0,\infty)$
if there exists a function $b\in\mathcal{D}(L^M)$, where $\mathcal{D}(L^M)$ denote the domian of $L^M$, such that $m=L^Mb$ and
\begin{align}\label{e2.1}
\lf\|\lf(r^2_BL\r)^kb\r\|_{L^2(U_j(B))}\le2^{-j\varepsilon}
r^{2M}_B\lf[\mu\lf(2^jB\r)\r]^{\frac1{2}}\lf\|\mathbf{1}_B\r\|^{-1}_{X(\mathbb{X})}
\end{align}
for any $k\in\{0,1,\ldots,M\}$ and $j\in\zz_+$.
\item[\rm{(ii)}] Assume that $\{m_{i,j}\}_{i\in\zz,j\in\nn}$ is a family of $(X,M,\varepsilon)_{L}$-molecules associated with the ball $\{B_{i,j}\}_{i\in\zz,j\in\nn}$ of $\mathbb{X}$ and numbers $\{\lambda_{i,j}\}_{i\in\zz,j\in\nn}$ satisfying
\begin{enumerate}
\item[\rm{(ii)}$_1$] for any $i\in\zz$ and $j\in\nn$, $\lambda_{i,j}:=2^i\|\mathbf{1}_{B_{i,j}}\|_{X(\mathbb{X})}$;
\item[\rm{(ii)}$_2$] there exists a positive constant $C$ such that
$$\sup_{i\in\zz}\Lambda\lf(\lf\{\lambda_{i,j}m_{i,j}\r\}_{i\in\zz,j\in\nn}\r):=
\sup_{i\in\zz}\lf\|\lf\{\sum_{j\in\nn}\lf[\frac{\lambda_{i,j}\mathbf{1}_{B_{i,j}}}
{\|\mathbf{1}_{B_{i,j}}\|_{X(\mathbb{X})}}\r]^{s_0}\r\}^{\frac1{s_0}}\r\|_{X(\mathbb{X})}<C;$$
\item[\rm{(ii)}$_3$] there exists a positive constant $C\in(0,1]$ such that, for any $x\in\mathbb{X}$ and $i\in\zz$,
$\sum_{j\in\nn}\mathbf{1}_{CB_{i,j}}(x)\leq M_0$, where $M_0$ being a positive constant independent of $x$ and $i$.
\end{enumerate}
Then, for any $f\in L^2(\mathbb{X})$, $f=\sum_{i\in\zz}\sum_{j\in\nn}\lambda_{i,j}m_{i,j}$
is called a {\it weak molecular $(X,M,\varepsilon)_{L}$-representation} of $f$ if the above summation converges in
$L^2(\mathbb{X})$. The {\it weak molecular Hardy space} $WH^{M,\varepsilon}_{X,L,\mathrm{mol}}(\mathbb{X})$ is defined
as the completion of the set
$$\mathbb{WH}^{M,\varepsilon}_{X,L,\mathrm{mol}}(\mathbb{X}):=\lf\{f\in L^2(\mathbb{X}):
\ f\ \mathrm{has\ a\ weak\ molecular}\ (X,M,\varepsilon)_{L}\emph{-} \mathrm{representation}\r\}$$
with respect to the {\it quasi-norm}
$$\|f\|_{WH^{M,\varepsilon}_{X,L,\mathrm{mol}}(\mathbb{X})}:
=\inf\lf\{\sup_{i\in\zz}\Lambda\lf(\lf\{\lambda_{i,j}m_{i,j}\r\}_{i\in\zz,j\in\nn}\r)\r\},$$
where the infimum is taken over all weak molecular $(X,M,\varepsilon)_{L}$-representations of $f$ as above.
\end{enumerate}
\end{definition}

\begin{definition}\label{de2.4}
Let $M\in\nn$ and $X(\mathbb{X})$ be a ${\rm{BQBF}}$ space satisfying Assumption \ref{as1.2}
for some $s_0\in(0,\infty)$ and $q_0\in(s_0,\infty)$. Assume that $L$ is a non-negative self-adjoint operator on
$L^2(\mathbb{X})$ satisfying the Davies-Gaffney estimate \eqref{e1.4}.
A measurable function $a\in L^2(\mathbb{X})$ is called an \emph{$(X,M)_{L}$-atom} associated with the ball $B:=B(x_B,r_B)\subset\mathbb{X}$ for some $x_B\in\mathbb{X}$ and $r_B\in(0,\infty)$
if there exists a function $b\in\mathcal{D}(L^M)$ such that $a=L^M(b)$, $\supp(L^k(b))\subset B$, and
\begin{align}\label{e2.2}
\lf\|\lf(r^2_BL\r)^kb\r\|_{L^2(\mathbb{X})}\le
r^{2M}_B\lf[\mu(B)\r]^{\frac1{2}}\lf\|\mathbf{1}_B\r\|^{-1}_{X(\mathbb{X})}
\end{align}
for any $k\in\{0,1,\ldots,M\}$. In a similar way, the set $\mathbb{WH}^M_{X,L,\mathrm{atom}}(\mathbb{X})$
and the weak atomic Hardy space $WH^M_{X,L,\mathrm{atom}}(\mathbb{X})$ are defined in the some way, respectively, as
$\mathbb{WH}^{M,\varepsilon}_{X,L,\mathrm{mol}}(\mathbb{X})$ and $WH^{M,\varepsilon}_{X,L,\mathrm{mol}}(\mathbb{X})$
with $(X,M,\varepsilon)_{L}$-molecules replaced by $(X,M)_{L}$-atoms.
\end{definition}

\section{Proofs of Theorem \ref{th1.2} and Corollary \ref{co1.2}}\label{s4}
\hskip\parindent
In this section, we give the proofs of Theorem \ref{th1.2} and Corollary \ref{co1.2}.
To show Theorem \ref{th1.2}, we need the following useful conclusion (see, for instance, \cite[Lemma 4.5]{lyyy24}).
\begin{lemma}\label{le4.1}
Let $t\in(0,\infty)$ and $0<\alpha<\beta<\infty$. Then
$$\sum_{\ell\in\zz}\lf(2^\ell t\r)^{-\alpha}\min\lf\{1,\lf(2^\ell t\r)^\beta\r\}
\leq\frac{1}{1-2^{\alpha-\beta}}+\frac{1}{1-2^{-\alpha}}.$$
\end{lemma}

Now, we show Theorem \ref{th1.2} via using Theorem \ref{th1.1}, Remark \ref{re2.2}, and Lemma \ref{le4.1}.
\begin{proof}[Proof of Theorem \ref{th1.2}]
Let $p\in(0,\infty)$, $s_0\in(0,\min\{p,1\}]$, $0<\gamma\neq1$,
$\beta\in[\gamma n(\frac{1}{s_0}-\frac{1}{2}),\infty)$, $k_0>\frac{\beta}{\gamma}$,
and $M>\max\{\frac{n}{2}(\frac{1}{s_0}-\frac{1}{2}),\frac{k_0}{2}\}$.
Moreover, assume that $\Psi\in C_c^\infty(\mathbb{R})$ is an even function such that $\supp(\Psi)\subset\{\xi:\frac1{4}\leq|\xi|\leq4\}$
and $\Psi=1$ on $\{\xi:\frac1{2}\leq|\xi|\leq2\}$. Without any confusion, we still denote
\begin{equation*}
S_{L}(f)(x):=\lf[\int^\infty_0\int_{d(x,y)<t}
\lf|\Psi\lf(t\sqrt{L}\r)(f)(y)\r|^2\frac{d\mu(y)dt}{V(x,t)t}\r]^{\frac1{2}}.
\end{equation*}
Recall that $\beta\in[\gamma n(\frac{1}{s_0}-\frac{1}{2}),\infty)$ and
$F_\tau(\xi)=(1+\xi^2)^{-\beta/2}e^{i\tau|\xi|^{\gamma/2}}$ such that
$F_\tau(\sqrt{L})=(I+L)^{-\beta/2}e^{i\tau L^{\gamma/2}}$.
Thus, by Remark \ref{re2.2}, to show Theorem \ref{th1.2}, it suffices to prove that,
for any $\tau\in \mathbb{R}$ and $f\in WH_{X,L}(\mathbb{X})\cap L^2(\mathbb{X})$,
\begin{align}\label{e4.1}
\lf\|S_L\lf(F_\tau\lf(\sqrt{L}\r)f\r)\r\|_{WX(\mathbb{X})}\lesssim \lf(1+|\tau|\r)^{n(\frac{1}{s_0}-\frac{r}{2})}\|f\|_{WH_{X,L}(\mathbb{X})}.
\end{align}

Let $f\in WH_{X,L}(\mathbb{X})\cap L^2(\mathbb{X})$. By Theorem \ref{th1.1}, we know that there exists a sequence
$\{a_{i,j}\}_{i\in\zz,j\in\nn}$ of $(X,M)_{L}$-atoms associated with the balls
$\{B_{i,j}\}_{i\in\zz,j\in\nn}$ such that
\begin{align}\label{e4.2}
f=\sum_{i\in\zz}\sum_{j\in\nn}\lambda_{i,j}a_{i,j} \ \ {\rm{in}} \ \ L^2(\mathbb{X}),
\end{align}
where $\lambda_{i,j}:=2^i\|\mathbf{1}_{B_{i,j}}\|_{X(\mathbb{X})}$ for each $i\in\zz$ and $j\in\nn$, and
\begin{align}\label{e4.3}
\sup_{i\in\zz}\lf\|\lf\{\sum_{j\in\nn}\lf[\frac{\lambda_{i,j}\mathbf{1}_{B_{i,j}}}
{\|\mathbf{1}_{B_{i,j}}\|_{X(\mathbb{X})}}\r]^{s_0}\r\}^{\frac1{s_0}}\r\|_{X(\mathbb{X})}
\lesssim\|f\|_{WH_{X,L}(\mathbb{X})}.
\end{align}
From \eqref{e4.2} and the fact that $S_L$ is bounded on $L^2(\mathbb{X})$, for almost every $x\in\mathbb{X}$,
\begin{align}\label{e4.4}
S_L\lf(F_\tau\lf(\sqrt{L}\r)f\r)(x)\leq\sum_{i\in\zz}\sum_{j\in\nn}\lambda_{i,j}S_L\lf(F_\tau\lf(\sqrt{L}\r)a_{i,j}\r)(x).
\end{align}
Moreover, it is easy to know that, for any $i\in\zz$ and $j\in\nn$,
$$I=\lf(I-e^{-r^2_{B_{i,j}}L}\r)^M+\sum^M_{k=1}(-1)^{k+1}C^k_Me^{-kr^2_{B_{i,j}}L}
=:\lf(I-e^{-r^2_{B_{i,j}}L}\r)^M+P\lf(r^2_{B_{i,j}}L\r),$$
where $I$ denotes the identity operator on $L^2(\mathbb{X})$ and, for any
$k\in\{0,1,\ldots,M\}$, $C^k_M:=\frac{M!}{k!(M-k)!}$, which implies that, for any $i\in\zz$, $j\in\nn$, and $x\in\mathbb{X}$,
\begin{align*}
S_L\lf(F_\tau\lf(\sqrt{L}\r)a_{i,j}\r)(x)&\leq S_L\lf(\lf[I-e^{-r^2_{B_{i,j}}L}\r]^MF_\tau\lf(\sqrt{L}\r)a_{i,j}\r)(x)\\
&\hs+S_L\lf(P\lf(r^2_{B_{i,j}}L\r)F_\tau\lf(\sqrt{L}\r)a_{i,j}\r)(x).
\end{align*}
From this, \eqref{e4.4}, and the definition of $WX(\mathbb{X})$, we deduce that
\begin{align}\label{e4.5}
\lf\|S_L\lf(F_\tau\lf(\sqrt{L}\r)f\r)\r\|_{WX(\mathbb{X})}
&=\sup_{\alpha\in(0,\infty)}\alpha\lf\|\mathbf{1}_{\{x\in\mathbb{X}: \ S_L(F_\tau(\sqrt{L})f)(x)>\alpha\}}\r\|_{X(\mathbb{X})}\nonumber\\
&\leq\sup_{\alpha\in(0,\infty)}\alpha\lf\|\mathbf{1}_{\{x\in\mathbb{X}: \ \sum_{i=-\infty}^{i_0-1}\sum_{j\in\nn}\lambda_{i,j}S_L([I-e^{-r^2_{B_{i,j}}L}]^M
F_\tau(\sqrt{L})a_{i,j})(x)>\frac{\alpha}{4}\}}\r\|_{X(\mathbb{X})}\nonumber\\
&\hs+\sup_{\alpha\in(0,\infty)}\alpha\lf\|\mathbf{1}_{\{x\in\mathbb{X}: \ \sum_{i=i_0}^\infty\sum_{j\in\nn}\lambda_{i,j}S_L([I-e^{-r^2_{B_{i,j}}L}]^M
F_\tau(\sqrt{L})a_{i,j})(x)>\frac{\alpha}{4}\}}\r\|_{X(\mathbb{X})}\nonumber\\
&\hs+\sup_{\alpha\in(0,\infty)}\alpha\lf\|\mathbf{1}_{\{x\in\mathbb{X}: \ \sum_{i=-\infty}^{i_0-1}\sum_{j\in\nn}\lambda_{i,j}S_L(P(r^2_{B_{i,j}}L)
F_\tau(\sqrt{L})a_{i,j})(x)>\frac{\alpha}{4}\}}\r\|_{X(\mathbb{X})}\nonumber\\
&\hs+\sup_{\alpha\in(0,\infty)}\alpha\lf\|\mathbf{1}_{\{x\in\mathbb{X}: \ \sum_{i=i_0}^\infty\sum_{j\in\nn}\lambda_{i,j}S_L(P(r^2_{B_{i,j}}L)
F_\tau(\sqrt{L})a_{i,j})(x)>\frac{\alpha}{4}\}}\r\|_{X(\mathbb{X})}\nonumber\\
&=:\sup_{\alpha\in(0,\infty)}\alpha\lf[{\rm{I_{1}+I_{2}+I_{3}+I_{4}}}\r].
\end{align}

Next, we show that, for any $(X,M)_{L}$-atom $a$, associated with the ball $B:=B(x_B,r_B)$ of $\mathbb{X}$ for some $x_B\in\mathbb{X}$ and $r_B\in(0,\infty)$, and any $k\in\zz_+$,
\begin{align}\label{e4.6}
\lf\|S_L\lf(\lf[I-e^{-r^2_{B}L}\r]^MF_\tau\lf(\sqrt{L}\r)a\r)\r\|_{L^2(U_k(B^\tau))}\ls2^{\frac{-k\beta}{\gamma}}
\lf[\mu\lf(B\r)\r]^{\frac1{2}}\lf\|\mathbf{1}_B\r\|^{-1}_{X(\mathbb{X})}
\end{align}
and
\begin{align}\label{e4.7}
\lf\|S_L\lf(P\lf(r^2_{B}L\r)F_\tau\lf(\sqrt{L}\r)a\r)\r\|_{L^2(U_k(B^\tau))}\ls2^{\frac{-k\beta}{\gamma}}
\lf[\mu\lf(B\r)\r]^{\frac1{2}}\lf\|\mathbf{1}_B\r\|^{-1}_{X(\mathbb{X})},
\end{align}
where $B^\tau:=(1+|\tau|)B$.

We first show \eqref{e4.6}. When $k\in\{0,1\}$, by the functional calculi associated with $L$ and \eqref{e2.2},
we know that
\begin{align}\label{e4.8}
\lf\|S_L\lf(\lf[I-e^{-r^2_{B}L}\r]^MF_\tau\lf(\sqrt{L}\r)a\r)\r\|_{L^2(U_k(B^\tau))}
\ls\|a\|_{L^2(\mathbb{X})}
\lesssim\lf[\mu\lf(B\r)\r]^{\frac1{2}}\lf\|\mathbf{1}_B\r\|^{-1}_{X(\mathbb{X})}.
\end{align}
When $k\in[2,\infty)\cap\nn$, for any $\lambda\in(0,\infty)$, setting
$$F_{t,\tau,r_B}(\lambda):=\Psi(t\lambda)\lf(I-e^{-r^2_{B}\lambda^2}\r)^MF_\tau(\lambda).$$
By this, $U_{k,t}(B^\tau):=\{x\in\mathbb{X}:\ d(x,U_{k}(B^\tau))\leq t\}$, and the fact that
\begin{align}\label{e4.9}
\int_{d(x,y)<t}\frac{d\mu(x)}{V(x,t)}\lesssim1,
\end{align}
we conclude that
\begin{align}\label{e4.10}
&\lf\|S_L\lf(\lf[I-e^{-r^2_{B}L}\r]^MF_\tau\lf(\sqrt{L}\r)a\r)\r\|^2_{L^2(U_k(B^\tau))}\nonumber\\
&\hs=\lf\|\lf[\int^\infty_0\int_{d(x,y)<t}\lf|\Psi\lf(t\sqrt{L}\r)\lf[I-e^{-r^2_{B}L}\r]^MF_\tau\lf(\sqrt{L}\r)a(y)\r|^2
\frac{d\mu(y)dt}{V(x,t)t}\r]^{\frac1{2}}\r\|^2_{L^2(U_k(B^\tau))}\nonumber\\
&\hs=\lf\|\lf[\int^\infty_0\int_{d(x,y)<t}\lf|F_{t,\tau,r_B}\lf(\sqrt{L}\r)a(y)\r|^2
\frac{d\mu(y)dt}{V(x,t)t}\r]^{\frac1{2}}\r\|^2_{L^2(U_k(B^\tau))}\nonumber\\
&\hs=\int_{U_k(B^\tau)}\int^\infty_0\int_{d(x,y)<t}\lf|F_{t,\tau,r_B}\lf(\sqrt{L}\r)a(y)\r|^2
\frac{d\mu(y)dt}{V(x,t)t}d\mu(x)\nonumber\\
&\hs\leq\int^\infty_0\int_{U_{k,t}(B^\tau)}\int_{d(x,y)<t}\frac{d\mu(x)}{V(x,t)}
\lf|F_{t,\tau,r_B}\lf(\sqrt{L}\r)a(y)\r|^2\frac{dt}{t}d\mu(y)\nonumber\\
&\hs\lesssim\int^\infty_0\int_{U_{k,t}(B^\tau)}
\lf|F_{t,\tau,r_B}\lf(\sqrt{L}\r)a(y)\r|^2d\mu(y)\frac{dt}{t}\nonumber\\
&\hs=\int^\infty_0\lf\|F_{t,\tau,r_B}\lf(\sqrt{L}\r)a\r\|^2_{L^2(U_{k,t}(B^\tau))}\frac{dt}{t}\nonumber\\
&\hs=\sum_{\ell\in\zz}\int^{2^{-\ell+1}}_{2^{-\ell}}
\lf\|F_{t,\tau,r_B}\lf(\sqrt{L}\r)a\r\|^2_{L^2(U_{k,t}(B^\tau))}\frac{dt}{t}\nonumber\\
&\hs\leq\sum_{\ell<0}\cdots+\sum_{0\leq\ell<\frac{k}{\gamma}}\cdots+\sum_{\ell\geq\frac{k}
{\gamma}}\cdots\nonumber\\
&\hs=:{\rm{II_1+II_2+II_3}}.
\end{align}
For the terms ${\rm{II_1}}$ and ${\rm{II_2}}$, by \cite[p.288]{bdd24}, we know that, for any $\ell\in\zz$, $k_0>\frac{\beta}{\gamma}$, and $g\in L^2(B)$,
$$\lf\|F_{t,\tau,r_B}\lf(\sqrt{L}\r)g\r\|_{L^2(U_{k,t}(B^\tau))}
\lesssim2^{-kk_0}\lf(2^\ell r_B\r)^{-k_0}\lf[1+2^{\gamma\ell(k_0-\frac{\beta}{\gamma})}\r]
\min\lf\{1,\lf(2^\ell r_B\r)^{2M}\r\}\|g\|_{L^2(B)}.$$
Applying this, Lemma \ref{le4.1} with $M>\frac{k_0}{2}$, \eqref{e2.2}, and $k_0>\frac{\beta}{\gamma}$, we conclude that
\begin{align}\label{e4.11}
{\rm{II_1}}&=\sum_{\ell<0}\int^{2^{-\ell+1}}_{2^{-\ell}}
\lf\|F_{t,\tau,r_B}\lf(\sqrt{L}\r)a\r\|^2_{L^2(U_{k,t}(B^\tau))}\frac{dt}{t}\nonumber\\
&\lesssim\sum_{\ell<0}\int^{2^{-\ell+1}}_{2^{-\ell}}2^{-2kk_0}\lf(2^\ell r_B\r)^{-2k_0}
\min\lf\{1,\lf(2^\ell r_B\r)^{4M}\r\}\|a\|^2_{L^2(B)}\frac{dt}{t}\nonumber\\
&\lesssim2^{-2kk_0}\sum_{\ell<0}\lf(2^\ell r_B\r)^{-2k_0}\min\lf\{1,\lf(2^\ell r_B\r)^{4M}\r\}\|a\|^2_{L^2(B)}\nonumber\\
&\lesssim2^{-2kk_0}\|a\|^2_{L^2(B)}
\ls\lf\{2^{\frac{-k\beta}{\gamma}}\lf[\mu\lf(B\r)\r]^{\frac1{2}}\lf\|\mathbf{1}_B\r\|^{-1}_{X(\mathbb{X})}\r\}^2
\end{align}
and
\begin{align}\label{e4.12}
{\rm{II_2}}&=\sum_{0\leq\ell<\frac{k}{\gamma}}\int^{2^{-\ell+1}}_{2^{-\ell}}
\lf\|F_{t,\tau,r_B}\lf(\sqrt{L}\r)a\r\|^2_{L^2(U_{k,t}(B^\tau))}\frac{dt}{t}\nonumber\\
&\lesssim\sum_{0\leq\ell<\frac{k}{\gamma}}\int^{2^{-\ell+1}}_{2^{-\ell}}2^{-2kk_0}\lf(2^\ell r_B\r)^{-2k_0}
2^{2\gamma\ell(k_0-\frac{\beta}{\gamma})}\min\lf\{1,\lf(2^\ell r_B\r)^{4M}\r\}\|a\|^2_{L^2(B)}\frac{dt}{t}\nonumber\\
&\lesssim2^{-2kk_0}2^{2k(k_0-\frac{\beta}{\gamma})}\sum_{0\leq\ell<\frac{k}{\gamma}}\lf(2^\ell r_B\r)^{-2k_0}\min\lf\{1,\lf(2^\ell r_B\r)^{4M}\r\}\|a\|^2_{L^2(B)}\nonumber\\
&\lesssim2^{\frac{-2k\beta}{\gamma}}\|a\|^2_{L^2(B)}
\ls\lf\{2^{\frac{-k\beta}{\gamma}}\lf[\mu\lf(B\r)\r]^{\frac1{2}}\lf\|\mathbf{1}_B\r\|^{-1}_{X(\mathbb{X})}\r\}^2.
\end{align}
For the term ${\rm{II_3}}$, by the fact that, for any $\ell\in\zz$,
$$\lf\|F_{t,\tau,r_B}\r\|_{L^\infty(0,\infty)}\lesssim2^{-\ell\beta}\min\lf\{1,\lf(2^\ell r_B\r)^{2M}\r\}$$
(see, for instance, \cite[p.287]{bdd24}), the functional calculi associated with $L$, and \eqref{e2.2}, we obtain that
\begin{align*}
{\rm{II_3}}&=\sum_{\ell\geq\frac{k}{\gamma}}\int^{2^{-\ell+1}}_{2^{-\ell}}
\lf\|F_{t,\tau,r_B}\lf(\sqrt{L}\r)a\r\|^2_{L^2(U_{k,t}(B^\tau))}\frac{dt}{t}\\
&\lesssim\sum_{\ell\geq\frac{k}{\gamma}}\int^{2^{-\ell+1}}_{2^{-\ell}}2^{-2\ell\beta}
\min\lf\{1,\lf(2^\ell r_B\r)^{4M}\r\}\|a\|^2_{L^2(B)}\frac{dt}{t}\\
&\lesssim\sum_{\ell\geq\frac{k}{\gamma}}2^{-2\ell\beta}\min\lf\{1,\lf(2^\ell r_B\r)^{4M}\r\}\|a\|^2_{L^2(B)}\\
&\lesssim2^{\frac{-2k\beta}{\gamma}}\|a\|^2_{L^2(B)}
\ls\lf\{2^{\frac{-k\beta}{\gamma}}\lf[\mu\lf(B\r)\r]^{\frac1{2}}\lf\|\mathbf{1}_B\r\|^{-1}_{X(\mathbb{X})}\r\}^2.
\end{align*}
From this, \eqref{e4.12}, \eqref{e4.11}, \eqref{e4.10}, and \eqref{e4.8}, we deduce that the claim \eqref{e4.6} holds true.

Now, we prove \eqref{e4.7}. For any $\ell\in\zz$ and $\lambda\in(0,\infty)$, let
$$G_{t,\tau,r_B}(\lambda):=\Psi(t\lambda)\lf(r^2_{B}\lambda^2\r)^MP\lf(r^2_{B}\lambda^2\r)F_\tau(\lambda).$$
By this, $a=L^Mb$, and \eqref{e4.9}, we know that
\begin{align}\label{e4.13}
&\lf\|S_L\lf(P\lf(r^2_{B}L\r)F_\tau\lf(\sqrt{L}\r)a\r)\r\|^2_{L^2(U_k(B^\tau))}\nonumber\\
&\hs=\lf\|S_L\lf(\lf(r^2_{B}L\r)^MP\lf(r^2_{B}L\r)F_\tau\lf(\sqrt{L}\r)r_B^{-2M}b\r)\r\|^2_{L^2(U_k(B^\tau))}\nonumber\\
&\hs=\lf\|\lf[\int^\infty_0\int_{d(x,y)<t}\lf|\Psi\lf(t\sqrt{L}\r)
\lf[\lf(r^2_{B}L\r)^MP\lf(r^2_{B}L\r)F_\tau\lf(\sqrt{L}\r)r_B^{-2M}b\r](y)\r|^2
\frac{d\mu(y)dt}{V(x,t)t}\r]^{\frac1{2}}\r\|^2_{L^2(U_k(B^\tau))}\nonumber\\
&\hs=\lf\|\lf[\int^\infty_0\int_{d(x,y)<t}\lf|G_{t,\tau,r_B}\lf(\sqrt{L}\r)r_B^{-2M}b(y)\r|^2
\frac{d\mu(y)dt}{V(x,t)t}\r]^{\frac1{2}}\r\|^2_{L^2(U_k(B^\tau))}\nonumber\\
&\hs=\int_{U_k(B^\tau)}\int^\infty_0\int_{d(x,y)<t}\lf|G_{t,\tau,r_B}\lf(\sqrt{L}\r)r_B^{-2M}b(y)\r|^2
\frac{d\mu(y)dt}{V(x,t)t}d\mu(x)\nonumber\\
&\hs\leq\int^\infty_0\int_{U_{k,t}(B^\tau)}\int_{d(x,y)<t}\frac{d\mu(x)}{V(x,t)}
\lf|G_{t,\tau,r_B}\lf(\sqrt{L}\r)r_B^{-2M}b(y)\r|^2\frac{dt}{t}d\mu(y)\nonumber\\
&\hs\lesssim\int^\infty_0\int_{U_{k,t}(B^\tau)}
\lf|G_{t,\tau,r_B}\lf(\sqrt{L}\r)r_B^{-2M}b(y)\r|^2d\mu(y)\frac{dt}{t}\nonumber\\
&\hs=r_B^{-4M}\int^\infty_0\lf\|G_{t,\tau,r_B}\lf(\sqrt{L}\r)b\r\|^2_{L^2(U_{k,t}(B^\tau))}\frac{dt}{t}.
\end{align}
Moreover, by checking the proof of \cite[Theorem 1.2]{bdd24}, we conclude that, for any $\ell\in\zz$, $k_0>\frac{\beta}{\gamma}$, and $g\in L^2(B)$,
$$\lf\|G_{t,\tau,r_B}\r\|_{L^\infty(0,\infty)}
\lesssim2^{-\ell\beta}\lf(2^\ell r_B\r)^{-2M}\min\lf\{1,\lf(2^\ell r_B\r)^{2M}\r\}$$
and
$$\lf\|G_{t,\tau,r_B}\lf(\sqrt{L}\r)g\r\|_{L^2(U_{k,t}(B^\tau))}
\lesssim2^{-kk_0}\lf(2^\ell r_B\r)^{-2M-k_0}\lf[1+2^{\gamma\ell(k_0-\frac{\beta}{\gamma})}\r]
\min\lf\{1,\lf(2^\ell r_B\r)^{2M}\r\}\|g\|_{L^2(B)}.$$
By this, \eqref{e4.13}, \eqref{e2.2}, and an argument similar to that used in the proof of \eqref{e4.6},
we conclude that the claim \eqref{e4.7} holds true.

Now, we deal with the term ${\rm{I_{1}}}$. Let $s_0\in(0,1]$, $q\in(1,\min\{\frac{1}{s_0},\frac{2}{q_0}\}]$,
and $a\in(0,1-\frac{1}{q})$. Then from the H\"{o}lder inequality, it follows that, for any $x\in\mathbb{X}$,
\begin{align*}
&\sum_{i=-\infty}^{i_0-1}\sum_{j\in\nn}\lambda_{i,j} S_L\lf(\lf[I-e^{-r^2_{B_{i,j}}L}\r]^MF_\tau\lf(\sqrt{L}\r)a_{i,j}\r)(x)\\
&\hs\leq\lf(\sum_{i=-\infty}^{i_0-1}2^{iaq'}\r)^{\frac{1}{q'}}
\lf\{\sum_{i=-\infty}^{i_0-1}2^{-iaq}\lf[\sum_{j\in\nn}\lambda_{i,j} S_L\lf(\lf[I-e^{-r^2_{B_{i,j}}L}\r]^MF_\tau\lf(\sqrt{L}\r)a_{i,j}\r)(x)\r]^q\r\}^{\frac{1}{q}}\\
&\hs=\frac{2^{i_0a}}{(2^{aq}-1)^{1/q}}
\lf\{\sum_{i=-\infty}^{i_0-1}2^{-iaq}\lf[\sum_{j\in\nn}\lambda_{i,j} S_L\lf(\lf[I-e^{-r^2_{B_{i,j}}L}\r]^MF_\tau\lf(\sqrt{L}\r)a_{i,j}\r)(x)\r]^q\r\}^{\frac{1}{q}}.
\end{align*}
By this, the assumption that $X^{\frac{1}{s_0}}(\mathbb{X})$ is a BBF space, the well-known inequality, for any $d\in(0,1]$
and $\{\theta_j\}_{j\in\nn}\subset\mathbb{C}$, $(\sum_{j\in\nn}|\theta_j|)^d\leq\sum_{j\in\nn}|\theta_j|^d$,
$q\in(1,\min\{\frac{1}{s_0},\frac{2}{q_0}\}]$, and $\lambda_{i,j}:=2^i\|\mathbf{1}_{B_{i,j}}\|_{X(\mathbb{X})}$, we know that
\begin{align}\label{e4.14}
{\rm{I_{1}}}&\lesssim\lf\|\mathbf{1}_{\{x\in\mathbb{X}: \ \frac{2^{i_0a}}{(2^{aq}-1)^{1/q}}
\{\sum_{i=-\infty}^{i_0-1}2^{-iaq}[\sum_{j\in\nn}\lambda_{i,j}S_L([I-e^{-r^2_{B_{i,j}}L}]^M
F_\tau(\sqrt{L})a_{i,j})(x)]^q\}^{\frac{1}{q}}>2^{i_0-2}\}}\r\|_{X(\mathbb{X})}\nonumber\\
&\lesssim2^{-i_0q(1-a)}\lf\|\sum_{i=-\infty}^{i_0-1}2^{-iaq}\lf[\sum_{j\in\nn}
\lambda_{i,j} S_L\lf(\lf[I-e^{-r^2_{B_{i,j}}L}\r]^MF_\tau\lf(\sqrt{L}\r)a_{i,j}\r)\r]^q\r\|_{X(\mathbb{X})}\nonumber\\
&\lesssim2^{-i_0q(1-a)}\lf\|\sum_{i=-\infty}^{i_0-1}2^{-iaqs_0}\sum_{k\in\zz_+}\lf[\sum_{j\in\nn}
\lambda_{i,j} S_L\lf(\lf[I-e^{-r^2_{B_{i,j}}L}\r]^MF_\tau\lf(\sqrt{L}\r)a_{i,j}\r)\mathbf{1}_{U_k(B^\tau_{i,j})}\r]^{qs_0}\r\|
^{\frac{1}{s_0}}_{X^{\frac{1}{s_0}}(\mathbb{X})}\nonumber\\
&\lesssim2^{-i_0q(1-a)}\lf\{\sum_{i=-\infty}^{i_0-1}\sum_{k\in\zz_+}2^{-iqs_0(a-1)}\r.\nonumber\\
&\hs\hs\hs\hs\hs\hs\lf.\times\lf\|\sum_{j\in\nn}\lf[\lf\|\mathbf{1}_{B_{i,j}}\r\|_{X(\mathbb{X})} S_L\lf(\lf[I-e^{-r^2_{B_{i,j}}L}\r]^MF_\tau\lf(\sqrt{L}\r)a_{i,j}\r)
\mathbf{1}_{U_k(B^\tau_{i,j})}\r]^{qs_0}\r\|_{X^{\frac{1}{s_0}}(\mathbb{X})}\r\}^{\frac{1}{s_0}},
\end{align}
where $B^\tau_{i,j}:=(1+|\tau|)B_{i,j}$ for any $i\in\zz$ and $j\in\nn$.

Moreover, from \eqref{e4.6}, it follows that, for any $i\in\zz$ and $j\in\nn$,
$$\lf\|\lf[\lf\|\mathbf{1}_{B_{i,j}}\r\|_{X(\mathbb{X})}S_L\lf(\lf[I-e^{-r^2_{B_{i,j}}L}\r]^MF_\tau
\lf(\sqrt{L}\r)a_{i,j}\r)\mathbf{1}_{U_k(B^\tau_{i,j})}\r]^{q}\r\|_{L^{\frac{2}{q}}
(\mathbb{X})}\ls2^{\frac{-k\beta q}{\gamma}}\lf[\mu\lf(B_{i,j}\r)\r]^{\frac{q}{2}}.$$
Applying this and Lemma \ref{le2.2} with $q:=\frac{2}{q}$ and $\theta:=2^k(1+|\tau|)$, we know that, for any $k\in\zz_+$,
\begin{align}\label{e4.15}
&\lf\|\sum_{j\in\nn}\lf[\lf\|\mathbf{1}_{B_{i,j}}\r\|_{X(\mathbb{X})}S_L\lf(\lf[I-e^{-r^2_{B_{i,j}}L}\r]^MF_\tau
\lf(\sqrt{L}\r)a_{i,j}\r)\mathbf{1}_{U_k(B^\tau_{i,j})}\r]^{qs_0}\r\|_{X^{\frac{1}{s_0}}(\mathbb{X})}\nonumber\\
&\hs\lesssim2^{k(1-\frac{s_0q}{2})n}(1+|\tau|)^{(1-\frac{s_0q}{2})n}\lf\|\sum_{j\in\nn}2^{\frac{-k\beta qs_0}{\gamma}} \mathbf{1}_{B_{i,j}}\r\|_{X^{\frac{1}{s_0}}(\mathbb{X})}\nonumber\\
&\hs=(1+|\tau|)^{(1-\frac{s_0q}{2})n}2^{-ks_0[\frac{\beta q}{\gamma}-(\frac{1}{s_0}-\frac{q}{2})n]}
\lf\|\sum_{j\in\nn}\mathbf{1}_{B_{i,j}}\r\|_{X^{\frac{1}{s_0}}(\mathbb{X})}.
\end{align}
Notice that $q>1$ and $\beta>\gamma n(\frac{1}{s_0}-\frac{1}{2})$, we conclude that $q>1>\frac{n/s_0}{(\beta/\gamma)+(n/2)}$,
which, together with \eqref{e4.14}, \eqref{e4.15}, $a\in(0,1-\frac{1}{q})$, $2^{i_0}\leq\alpha<2^{i_0+1}$,
$\lambda_{i,j}:=2^i\|\mathbf{1}_{B_{i,j}}\|_{X(\mathbb{X})}$, the assumption that $X^{\frac{1}{s_0}}(\mathbb{X})$
is a BBF space, Remark \ref{re2.1}, and \eqref{e4.3}, further implies that
\begin{align}\label{e4.16}
{\rm{I_{1}}}&\lesssim2^{-i_0q(1-a)}\lf\{\sum_{i=-\infty}^{i_0-1}\sum_{k\in\zz_+}2^{-iqs_0(a-1)}
(1+|\tau|)^{(1-\frac{s_0q}{2})n}2^{-ks_0[\frac{\beta q}{\gamma}-(\frac{1}{s_0}-\frac{q}{2})n]}
\lf\|\sum_{j\in\nn}\mathbf{1}_{B_{i,j}}\r\|_{X^{\frac{1}{s_0}}(\mathbb{X})}\r\}^{\frac{1}{s_0}}\nonumber\\
&\lesssim2^{-i_0q(1-a)}(1+|\tau|)^{(\frac{1}{s_0}-\frac{q}{2})n}\lf\{\sum_{i=-\infty}^{i_0-1}2^{-iqs_0(a-1)}
\lf\|\sum_{j\in\nn}\mathbf{1}_{B_{i,j}}\r\|_{X^{\frac{1}{s_0}}(\mathbb{X})}\r\}^{\frac{1}{s_0}}\nonumber\\
&\lesssim2^{-i_0q(1-a)}(1+|\tau|)^{(\frac{1}{s_0}-\frac{q}{2})n}
\lf\{\sum_{i=-\infty}^{i_0-1}2^{-iqs_0(a-1)}2^{-is_0}\r\}^{\frac{1}{s_0}}
\sup_{i\in\zz}2^i\lf\|\sum_{j\in\nn}\mathbf{1}_{B_{i,j}}\r\|^{\frac{1}{s_0}}_{X^{\frac{1}{s_0}}(\mathbb{X})}\nonumber\\
&\lesssim2^{-i_0}(1+|\tau|)^{(\frac{1}{s_0}-\frac{q}{2})n}\sup_{i\in\zz}2^i\lf\|\sum_{j\in\nn}\mathbf{1}_{B_{i,j}}\r\|
^{\frac{1}{s_0}}_{X^{\frac{1}{s_0}}(\mathbb{X})}\nonumber\\
&\lesssim\alpha^{-1}(1+|\tau|)^{(\frac{1}{s_0}-\frac{1}{2})n}\sup_{i\in\zz}2^i\lf\|\lf[\sum_{j\in\nn}\mathbf{1}_{B_{i,j}}\r]
^{\frac{1}{s_0}}\r\|_{X(\mathbb{X})}\nonumber\\
&\lesssim\alpha^{-1}(1+|\tau|)^{(\frac{1}{s_0}-\frac{1}{2})n}\sup_{i\in\zz}2^i\lf\|\sum_{j\in\nn}\mathbf{1}_{B_{i,j}}
\r\|_{X(\mathbb{X})}\nonumber\\
&\lesssim\alpha^{-1}(1+|\tau|)^{(\frac{1}{s_0}-\frac{1}{2})n}\sup_{i\in\zz}\lf\|\sum_{j\in\nn}\frac{\lambda_{i,j}\mathbf{1}_{B_{i,j}}}
{\|\mathbf{1}_{B_{i,j}}\|_{X(\mathbb{X})}}\r\|_{X(\mathbb{X})}\nonumber\\
&\lesssim\alpha^{-1}(1+|\tau|)^{(\frac{1}{s_0}-\frac{1}{2})n}\sup_{i\in\zz}\lf\|\lf\{\sum_{j\in\nn}\lf[\frac{\lambda_{i,j}\mathbf{1}_{B_{i,j}}}
{\|\mathbf{1}_{B_{i,j}}\|_{X(\mathbb{X})}}\r]^{s_0}\r\}^{\frac{1}{s_0}}\r\|_{X(\mathbb{X})}\nonumber\\
&\lesssim\alpha^{-1}(1+|\tau|)^{(\frac{1}{s_0}-\frac{1}{2})n}\|f\|_{WH_{X,L}(\mathbb{X})}.
\end{align}

Next, we deal with the term ${\rm{I_{2}}}$. Since $\beta>\gamma n(\frac{1}{s_0}-\frac{1}{2})$, it follows that there exists $r\in(0,1)$ such that $\beta>\gamma n(\frac{1}{rs_0}-\frac{1}{2})$. Then, by the assumption that $X^{\frac{1}{s_0}}(\mathbb{X})$
is a BBF space, we know that
\begin{align}\label{e4.17}
{\rm{I_{2}}}&=\lf\|\mathbf{1}_{\{x\in\mathbb{X}: \ \sum_{i=i_0}^\infty\sum_{j\in\nn}\lambda_{i,j}S_L([I-e^{-r^2_{B_{i,j}}L}]^M
F_\tau(\sqrt{L})a_{i,j})(x)>\frac{\alpha}{4}\}}\r\|_{X(\mathbb{X})}\nonumber\\
&\lesssim\alpha^{-r}\lf\|\sum_{i=i_0}^\infty\sum_{j\in\nn}\lf[2^i\lf\|\mathbf{1}_{B_{i,j}}\r\|_{X(\mathbb{X})}
S_L\lf(\lf[I-e^{-r^2_{B_{i,j}}L}\r]^MF_\tau\lf(\sqrt{L}\r)a_{i,j}\r)\r]^r\r\|_{X(\mathbb{X})}\nonumber\\
&\lesssim\alpha^{-r}\lf\|\sum_{i=i_0}^\infty\sum_{j\in\nn}\sum_{k\in\zz_+}\lf[2^i\lf\|\mathbf{1}_{B_{i,j}}\r\|_{X(\mathbb{X})}
S_L\lf(\lf[I-e^{-r^2_{B_{i,j}}L}\r]^MF_\tau\lf(\sqrt{L}\r)a_{i,j}\r)
\mathbf{1}_{U_k(B^\tau_{i,j})}\r]^{rs_0}\r\|^{\frac{1}{s_0}}_{X^{\frac{1}{s_0}}(\mathbb{X})}\nonumber\\
&\lesssim\alpha^{-r}\lf\{\sum_{i=i_0}^\infty\sum_{k\in\zz_+}2^{irs_0}\r.\nonumber\\
&\hs\hs\hs\hs\lf.\times\lf\|\sum_{j\in\nn}\lf[\lf\|\mathbf{1}_{B_{i,j}}\r\|_{X(\mathbb{X})}
S_L\lf(\lf[I-e^{-r^2_{B_{i,j}}L}\r]^MF_\tau\lf(\sqrt{L}\r)a_{i,j}\r)
\mathbf{1}_{U_k(B^\tau_{i,j})}\r]^{rs_0}\r\|_{X^{\frac{1}{s_0}}(\mathbb{X})}\r\}^{\frac{1}{s_0}}.
\end{align}
Moreover, from \eqref{e4.6}, $r\in(0,1)$, it follows that,
for any $i\in\zz$ and $j\in\nn$,
$$\lf\|\lf[\lf\|\mathbf{1}_{B_{i,j}}\r\|_{X(\mathbb{X})}
S_L\lf(\lf[I-e^{-r^2_{B_{i,j}}L}\r]^MF_\tau\lf(\sqrt{L}\r)a_{i,j}\r)
\mathbf{1}_{U_k(B^\tau_{i,j})}\r]^{r}\r\|_{L^{\frac{2}{r}}
(\mathbb{X})}\ls2^{\frac{-k\beta r}{\gamma}}\lf[\mu\lf(B_{i,j}\r)\r]^{\frac{r}{2}}.$$
By this, \eqref{e4.15} with $q$ there in replaced by $r$, $\beta>\gamma n(\frac{1}{rs_0}-\frac{1}{2})$, $2^{i_0}\leq\alpha<2^{i_0+1}$,
$\lambda_{i,j}:=2^i\|\mathbf{1}_{B_{i,j}}\|_{X(\mathbb{X})}$, the assumption that $X^{\frac{1}{s_0}}(\mathbb{X})$
is a BBF space, Remark \ref{re2.1}, \eqref{e4.3}, and similar to the proof of \eqref{e4.16}, we conclude that
\begin{align}\label{e4.18}
{\rm{I_{2}}}&\lesssim\alpha^{-r}\lf\{\sum_{i=i_0}^\infty\sum_{k\in\zz_+}2^{irs_0}
(1+|\tau|)^{(1-\frac{rs_0}{2})n}2^{-ks_0[\frac{\beta r}{\gamma}-(\frac{1}{s_0}-\frac{r}{2})n]}
\lf\|\sum_{j\in\nn}\mathbf{1}_{B_{i,j}}\r\|_{X^{\frac{1}{s_0}}(\mathbb{X})}\r\}^{\frac{1}{s_0}}\nonumber\\
&\lesssim\alpha^{-r}\lf\{\sum_{i=i_0}^\infty2^{irs_0}(1+|\tau|)^{(1-\frac{rs_0}{2})n}
\lf\|\sum_{j\in\nn}\mathbf{1}_{B_{i,j}}\r\|_{X^{\frac{1}{s_0}}(\mathbb{X})}\r\}^{\frac{1}{s_0}}\nonumber\\
&\lesssim\alpha^{-r}(1+|\tau|)^{(\frac1{s_0}-\frac{r}{2})n}\lf\{\sum_{i=i_0}^\infty2^{irs_0}2^{-is_0}\r\}^{\frac{1}{s_0}}
\sup_{i\in\zz}2^i\lf\|\sum_{j\in\nn}\mathbf{1}_{B_{i,j}}\r\|^{\frac{1}{s_0}}_{X^{\frac{1}{s_0}}(\mathbb{X})}\nonumber\\
&\lesssim\alpha^{-r}(1+|\tau|)^{(\frac1{s_0}-\frac{r}{2})n}2^{-i_0(1-r)}
\sup_{i\in\zz}2^i\lf\|\sum_{j\in\nn}\mathbf{1}_{B_{i,j}}\r\|^{\frac{1}{s_0}}_{X^{\frac{1}{s_0}}(\mathbb{X})}\nonumber\\
&\lesssim\alpha^{-1}(1+|\tau|)^{(\frac1{s_0}-\frac{r}{2})n}\|f\|_{WH_{X,L}(\mathbb{X})}.
\end{align}

Finally, we deal with the terms ${\rm{I_{3}}}$ and ${\rm{I_{4}}}$.
By \eqref{e4.7} and arguments similar to the proofs of ${\rm{I_{1}}}$ and ${\rm{I_{2}}}$,
we obtain that
\begin{align}\label{e4.19}
{\rm{I_{3}}}\lesssim\alpha^{-1}(1+|\tau|)^{(\frac1{s_0}-\frac{1}{2})n}\|f\|_{WH_{X,L}(\mathbb{X})}
\end{align}
and
\begin{align}\label{e4.20}
{\rm{I_{4}}}\lesssim\alpha^{-1}(1+|\tau|)^{(\frac1{s_0}-\frac{r}{2})n}\|f\|_{WH_{X,L}(\mathbb{X})},
\end{align}
where $r\in(0,1)$, the details being omitted here.
Combining the above estimates of \eqref{e4.20}, \eqref{e4.19}, \eqref{e4.18}, \eqref{e4.16}, and \eqref{e4.5},
we conclude that \eqref{e4.1} holds true.
This completes the proof of Theorem \ref{th1.2}.
\end{proof}

Next, we show Corollary \ref{co1.2} via using Theorem \ref{co1.1}.
\begin{proof}[Proof of Corollary \ref{co1.2}]
Let $0<\gamma\neq1$, $p:=r$, $q_0:=2$, and $s_0\in(0,\min\{r,1\})$ satisfying
$\beta\in[\gamma n(\frac{1}{s_0}-\frac{1}{2}),\infty)$.
In the case when $X(\mathbb{X}):=L^r(\mathbb{X})$, we denote $H_{X,L}(\mathbb{X})$ by $H^r_{L}(\mathbb{X})$.
From \cite[Theorem 1.2]{gly09}, we deduce that $X(\mathbb{X}):=L^r(\mathbb{X})$
satisfies Assumption \ref{as1.1} for the aforementioned $p$.
Moreover, by \cite[p.10, Theorem 2.5]{bs88} and the boundedness of $\mathcal{M}$ on $L^q(\mathbb{X})$
for any given $q\in(1,\infty)$ (see, for instance, \cite[Theorem 2.2]{h01}), we know that
$X(\mathbb{X}):=L^r(\mathbb{X})$ satisfies Assumption \ref{as1.2} for the aforementioned $s_0$ and $q_0$.
Thus, all the assumptions of Theorem \ref{co1.1} are satisfied with
$X(\mathbb{X}):=L^r(\mathbb{X})$, which consequently yields the desired conclusions of the present theorem.
This completes the proof of Corollary \ref{co1.2}.
\end{proof}




\section{Proof of Theorem \ref{th1.3}}\label{s5}
\hskip\parindent
In this section, we give the proof of Theorem \ref{th1.3}.
To show Theorem \ref{th1.3}, we state the following lemma.

Let $\psi\in\cs(\mathbb{R})$ be such that $\supp(\psi)\subset[\frac1{2},2]$, $\int_{\mathbb{R}}\psi(t)\frac{dt}{t}\neq0$, and
$$\sum_{j\in\zz}\psi\lf(2^{-2j}\lambda\r)=1$$ for any $\lambda\in(0,\infty)$. For any $\lambda\in(0,\infty)$.
the {\it discrete square function} $S_{L,\psi}(f)$ is defined by setting, for any $x\in\mathbb{X}$,
\begin{equation*}
S_{L,\psi}(f)(x):=\lf[\sum_{j\in\zz}\lf|\psi\lf(2^{-2j}L\r)(f)(x)\r|^2\r]^{\frac1{2}}.
\end{equation*}
Moreover, recall that, for any $t,\lambda\in(0,\infty)$ and $\varphi\in\cs(\mathbb{R})$,
the {\it peetre type maximal operator} $\varphi^*_\lambda(tL)$ is defined by setting,
for any $f\in L^2(\mathbb{X})$ and $x\in\mathbb{X}$,
$$\varphi^*_\lambda(tL)(f)(x):=\sup_{x\in\mathbb{X}}\frac{|\varphi(tL)f(y)|}{[1+t^{-1/2}d(x,y)]^\lambda}.$$

\begin{lemma}\label{le5.1}
Let $L$ be a non-negative self-adjoint operator on $L^2(\mathbb{X})$
satisfying the Gaussian upper estimate \eqref{e1.7}.
Assume that $X(\mathbb{X})$ is a ${\rm{BQBF}}$ space satisfying Assumption \ref{as1.1} for some $p\in(0,\infty)$.
Then there exists a positive constant $C$ such that, for any $f\in L^2(\mathbb{X})$
with $\|S_{L,\psi}(f)\|_{WX(\mathbb{X})}<\infty$,
$$\|f\|_{WX_{X,L}(\mathbb{X})}\leq C\lf\|S_{L,\psi}(f)\r\|_{WX(\mathbb{X})}.$$
\end{lemma}

\begin{proof}
Let $\lambda\in(\frac{n}{p},\infty)$, $r\in(0,\min\{p,1\})$, $\varphi(\xi):=\xi e^{-\xi}$ for any $\xi\in\mathbb{R}$,
and $f\in L^2(\mathbb{X})$ with $\|S_{L,\psi}(f)\|_{WX(\mathbb{X})}<\infty$. From \cite[p.8 and p.11]{bl22},
it follows that, for any $j\in\zz$ and $x\in\mathbb{X}$,
$$\psi^*_\lambda\lf(2^{-2j}L\r)f(x)\lesssim\lf[\mathcal{M}\lf(\lf|\psi\lf(2^{-2j}L\r)f\r|^r\r)(x)\r]^{\frac{1}{r}}$$
and
$$\lf[\int^\infty_0\lf|\varphi^*_\lambda\lf(t^2L\r)f(x)\r|^2\frac{dt}{t}\r]^{\frac1{2}}
\lesssim\lf\{\sum_{j\in\zz}\lf[\psi^*_\lambda\lf(2^{-2j}L\r)f(x)\r]^2\r\}^{\frac1{2}}.$$
By this, the fact that, for any $t\in(0,\infty)$ and $x,y\in\mathbb{X}$ satisfying $d(x,y)<t$,
$$\lf|\varphi\lf(t^2L\r)f(y)\r|\leq2^\lambda\varphi^*_\lambda\lf(t^2L\r)f(x),$$
and Lemma \ref{le2.3}, we obtain that
\begin{align*}
\|f\|_{WH_{X,L}(\mathbb{X})}&=\lf\|S_L(f)\r\|_{WX(\mathbb{X})}\\
&=\lf\|\lf[\int^\infty_0\int_{B(\cdot,t)}
\lf|t^2 L e^{-t^2 L}(f)(y)\r|^2\frac{d\mu(y)dt}{V(\cdot,t)t}\r]^{\frac1{2}}\r\|_{WX(\mathbb{X})}\\
&=\lf\|\lf[\int^\infty_0\int_{B(\cdot,t)}
\lf|\varphi\lf(t^2L\r)f(y)\r|^2\frac{d\mu(y)dt}{V(\cdot,t)t}\r]^{\frac1{2}}\r\|_{WX(\mathbb{X})}\\
&\lesssim\lf\|\lf[\int^\infty_0\lf|\varphi^*_\lambda\lf(t^2L\r)f\r|^2\frac{dt}{t}\r]^{\frac1{2}}\r\|_{WX(\mathbb{X})}\\
&\lesssim\lf\|\lf\{\sum_{j\in\zz}\lf[\mathcal{M}
\lf(\lf|\psi\lf(2^{-2j}L\r)f\r|^r\r)\r]^{\frac{2}{r}}\r\}^{\frac1{2}}\r\|_{WX(\mathbb{X})}\\
&\lesssim\lf\|\lf[\sum_{j\in\zz}\lf|\psi\lf(2^{-2j}L\r)f\r|^2\r]^{\frac1{2}}\r\|_{WX(\mathbb{X})}
=\lf\|S_{L,\psi}(f)\r\|_{WX(\mathbb{X})}.
\end{align*}
This completes the proof of Lemma \ref{le5.1}.
\end{proof}

Now, we show Theorem \ref{th1.3} via using Theorem \ref{th1.1} and Lemma \ref{le5.1}.
\begin{proof}[Proof of Theorem \ref{th1.3}]
Let $p\in(0,\infty)$, $s_0\in(0,\min\{p,1\}]$, $\alpha>n(\frac{1}{s_0}-\frac{1}{2})>\frac{n}{2}$, and
$M>\max\{\frac{n}{2}(\frac{1}{s_0}-\frac{1}{2}),\frac{1}{2}(\alpha-\frac{n}{2})\}$.
By Lemma \ref{le5.1}, we conclude that, for any $\tau\in \mathbb{R}$ and $f\in WH_{X,L}(\mathbb{X})\cap L^2(\mathbb{X})$,
$$\lf\|L^{i\tau}f\r\|_{WH_{X,L}(\mathbb{X})}\lesssim\lf\|S_{L,\psi}\lf(L^{i\tau}f\r)\r\|_{WX(\mathbb{X})}.$$
Thus, to show Theorem \ref{th1.3}, it suffices to show that, for any $\tau\in \mathbb{R}$ and $f\in WH_{X,L}(\mathbb{X})\cap L^2(\mathbb{X})$,
\begin{align}\label{e5.1}
\lf\|S_{L,\psi}\lf(L^{i\tau}f\r)\r\|_{WX(\mathbb{X})}\lesssim
\lf(1+|\tau|\r)^{n(\frac{1}{s_0}-\frac{r}{2})}\|f\|_{WH_{X,L}(\mathbb{X})}.
\end{align}
Let $f\in WH_{X,L}(\mathbb{X})\cap L^2(\mathbb{X})$. By Theorem \ref{th1.1}, we know that there exists a sequence
$\{a_{i,j}\}_{i\in\zz,j\in\nn}$ of $(X,M)_{L}$-atoms associated with the balls
$\{B_{i,j}\}_{i\in\zz,j\in\nn}$ such that
\begin{align}\label{e5.2}
f=\sum_{i\in\zz}\sum_{j\in\nn}\lambda_{i,j}a_{i,j} \ \ {\rm{in}} \ \ L^2(\mathbb{X}),
\end{align}
where $\lambda_{i,j}:=2^i\|\mathbf{1}_{B_{i,j}}\|_{X(\mathbb{X})}$ for each $i\in\zz$ and $j\in\nn$, and
\begin{align}\label{e5.3}
\sup_{i\in\zz}\lf\|\lf\{\sum_{j\in\nn}\lf[\frac{\lambda_{i,j}\mathbf{1}_{B_{i,j}}}
{\|\mathbf{1}_{B_{i,j}}\|_{X(\mathbb{X})}}\r]^{s_0}\r\}^{\frac1{s_0}}\r\|_{X(\mathbb{X})}
\lesssim\|f\|_{WH_{X,L}(\mathbb{X})}.
\end{align}
By \eqref{e5.2} and the fact that $S_{L,\psi}$ is bounded on $L^2(\mathbb{X})$, for almost every $x\in\mathbb{X}$,
\begin{align}\label{e5.4}
S_{L,\psi}\lf(L^{i\tau}f\r)(x)\leq\sum_{i\in\zz}\sum_{j\in\nn}\lambda_{i,j}S_{L,\psi}\lf(L^{i\tau}a_{i,j}\r)(x).
\end{align}
Moreover, by the spectral theory, we know that, for any $i\in\zz$ and $j\in\nn$,
$$I=\lf(I-e^{-r^2_{B_{i,j}}L}\r)^M+\sum^M_{k=1}(-1)^{k+1}C^k_Me^{-kr^2_{B_{i,j}}L}
=:\lf(I-e^{-r^2_{B_{i,j}}L}\r)^M+P\lf(r^2_{B_{i,j}}L\r),$$
where $I$ denotes the identity operator on $L^2(\mathbb{X})$ and, for any
$k\in\{0,1,\ldots,M\}$, $C^k_M:=\frac{M!}{k!(M-k)!}$, which implies that, for any $i\in\zz$, $j\in\nn$, and $x\in\mathbb{X}$,
\begin{align*}
S_{L,\psi}\lf(L^{i\tau}a_{i,j}\r)(x)&\leq S_{L,\psi}\lf(\lf[I-e^{-r^2_{B_{i,j}}L}\r]^ML^{i\tau}a_{i,j}\r)(x)+
S_{L,\psi}\lf(P\lf(r^2_{B_{i,j}}L\r)L^{i\tau}a_{i,j}\r)(x).
\end{align*}
By this, \eqref{e5.4}, and the definition of $WX(\mathbb{X})$, we conclude that
\begin{align}\label{e5.5}
\lf\|S_{L,\psi}\lf(L^{i\tau}f\r)\r\|_{WX(\mathbb{X})}
&=\sup_{\alpha\in(0,\infty)}\alpha\lf\|\mathbf{1}_{\{x\in\mathbb{X}: \ S_{L,\psi}(L^{i\tau}f)(x)>\alpha\}}\r\|_{X(\mathbb{X})}\nonumber\\
&\leq\sup_{\alpha\in(0,\infty)}\alpha\lf\|\mathbf{1}_{\{x\in\mathbb{X}: \ \sum_{i=-\infty}^{i_0-1}\sum_{j\in\nn}\lambda_{i,j}S_{L,\psi}([I-e^{-r^2_{B_{i,j}}L}]^M
L^{i\tau}a_{i,j})(x)>\frac{\alpha}{4}\}}\r\|_{X(\mathbb{X})}\nonumber\\
&\hs+\sup_{\alpha\in(0,\infty)}\alpha\lf\|\mathbf{1}_{\{x\in\mathbb{X}: \ \sum_{i=i_0}^\infty\sum_{j\in\nn}\lambda_{i,j}S_{L,\psi}([I-e^{-r^2_{B_{i,j}}L}]^M
L^{i\tau}a_{i,j})(x)>\frac{\alpha}{4}\}}\r\|_{X(\mathbb{X})}\nonumber\\
&\hs+\sup_{\alpha\in(0,\infty)}\alpha\lf\|\mathbf{1}_{\{x\in\mathbb{X}: \ \sum_{i=-\infty}^{i_0-1}\sum_{j\in\nn}\lambda_{i,j}S_{L,\psi}(P(r^2_{B_{i,j}}L)
L^{i\tau}a_{i,j})(x)>\frac{\alpha}{4}\}}\r\|_{X(\mathbb{X})}\nonumber\\
&\hs+\sup_{\alpha\in(0,\infty)}\alpha\lf\|\mathbf{1}_{\{x\in\mathbb{X}: \ \sum_{i=i_0}^\infty\sum_{j\in\nn}\lambda_{i,j}S_{L,\psi}(P(r^2_{B_{i,j}}L)
L^{i\tau}a_{i,j})(x)>\frac{\alpha}{4}\}}\r\|_{X(\mathbb{X})}\nonumber\\
&=:\sup_{\alpha\in(0,\infty)}\alpha\lf[{\rm{II_{1}+II_{2}+II_{3}+II_{4}}}\r].
\end{align}

Next, we prove that, for any $(X,M)_{L}$-atom $a$, associated with the ball $B:=B(x_B,r_B)$ of $\mathbb{X}$ for some $x_B\in\mathbb{X}$ and $r_B\in(0,\infty)$, and any $k\in\zz_+$ and $\alpha>\frac{n}{2}$,
\begin{align}\label{e5.6}
\lf\|S_{L,\psi}\lf(\lf[I-e^{-r^2_{B}L}\r]^ML^{i\tau}a\r)\r\|_{L^2(U_k(B^\tau))}\ls2^{-k\alpha}
\lf[\mu\lf(B\r)\r]^{\frac1{2}}\lf\|\mathbf{1}_B\r\|^{-1}_{X(\mathbb{X})}
\end{align}
and
\begin{align}\label{e5.7}
\lf\|S_{L,\psi}\lf(P\lf(r^2_{B}L\r)L^{i\tau}a\r)\r\|_{L^2(U_k(B^\tau))}\ls2^{-k\alpha}
\lf[\mu\lf(B\r)\r]^{\frac1{2}}\lf\|\mathbf{1}_B\r\|^{-1}_{X(\mathbb{X})},
\end{align}
where $B^\tau:=(1+|\tau|)B$.

We first show \eqref{e5.6}. If $k\in\{0,1\}$, by the $L^2$ boundedness of
$S_{L,\psi}$ and $L^{i\tau}$, and \eqref{e2.2}, we know that
\begin{align}\label{e5.8}
\lf\|S_{L,\psi}\lf(\lf[I-e^{-r^2_{B}L}\r]^ML^{i\tau}a\r)\r\|_{L^2(U_k(B^\tau))}
\ls\|a\|_{L^2(\mathbb{X})}
\lesssim\lf[\mu\lf(B\r)\r]^{\frac1{2}}\lf\|\mathbf{1}_B\r\|^{-1}_{X(\mathbb{X})}.
\end{align}
If $k\in[2,\infty)\cap\nn$, for any $\ell\in\zz$ and $\lambda\in(0,\infty)$, setting
$$F_{\ell,\tau,r_B}(\lambda):=\psi\lf(2^{-2\ell}\lambda\r)\lf(I-e^{-r^2_{B}\lambda}\r)^M\lambda^{i\tau}.$$
By \cite[(3.7)]{bbhh22}, we know that, for any $\ell\in\zz$, $\alpha\in2\nn$ with $\alpha>\frac{n}{2}$, and $g\in L^2(B)$,
$$\lf\|F_{\ell,\tau,r_B}(L)g\r\|_{L^2(U_{k}(B^\tau))}
\lesssim2^{-k\alpha}\lf(2^\ell r_B\r)^{-(\alpha-\frac{n}{2})}\min\lf\{1,\lf(2^\ell r_B\r)^{2M}\r\}\|g\|_{L^2(B)}.$$
Applying this, the well-known inequality, for any $d\in(0,1]$
and $\{\theta_j\}_{j\in\nn}\subset\mathbb{C}$, $(\sum_{j\in\nn}|\theta_j|)^d\leq\sum_{j\in\nn}|\theta_j|^d$, Lemma \ref{le4.1} with $M>\frac{1}{2}(\alpha-\frac{n}{2})$, \eqref{e2.2}, we conclude that, for any $\alpha\in2\nn$ with $\alpha>\frac{n}{2}$,
and any $k\in[2,\infty)\cap\nn$,
\begin{align*}
&\lf\|S_{L,\psi}\lf(\lf[I-e^{-r^2_{B}L}\r]^ML^{i\tau}a\r)\r\|_{L^2(U_k(B^\tau))}\\
&\hs=\lf\|\lf[\sum_{\ell\in\zz}\lf|F_{\ell,\tau,r_B}(L)a\r|^2\r]^{\frac{1}{2}}\r\|_{L^2(U_k(B^\tau))}
\leq\sum_{\ell\in\zz}\lf\|F_{\ell,\tau,r_B}(L)a\r\|_{L^2(U_k(B^\tau))}\\
&\hs\lesssim\sum_{\ell\in\zz}2^{-k\alpha}\lf(2^\ell r_B\r)^{-(\alpha-\frac{n}{2})}\min\lf\{1,\lf(2^\ell r_B\r)^{2M}\r\}\|a\|_{L^2(B)}\\
&\hs\lesssim2^{-k\alpha}\|a\|_{L^2(B)}
\ls2^{-k\alpha}\lf[\mu\lf(B\r)\r]^{\frac1{2}}\lf\|\mathbf{1}_B\r\|^{-1}_{X(\mathbb{X})},
\end{align*}
which combined with \eqref{e5.8}, further implies that \eqref{e5.6} holds true.

Next, we show \eqref{e5.7}.
For any $\ell\in\zz$ and $\lambda\in(0,\infty)$, let
$$G_{\ell,\tau,r_B}(\lambda):=\psi\lf(2^{-2\ell}\lambda\r)\lf(r^2_{B}\lambda\r)^MP\lf(r^2_{B}\lambda\r)\lambda^{i\tau}.$$
By \cite[pp.6-7]{bbhh22}, we know that, for any $\ell\in\zz$, $\alpha\in2\nn$ with $\alpha>\frac{n}{2}$, and $g\in L^2(B)$,
\begin{align}\label{e5.9}
\lf\|G_{\ell,\tau,r_B}(L)g\r\|_{L^2(U_{k}(B^\tau))}
\lesssim2^{-k\alpha}\lf(2^\ell r_B\r)^{-(\alpha-\frac{n}{2})}
\min\lf\{\lf(2^\ell r_B\r)^{-2M},\lf(2^\ell r_B\r)^{2M}\r\}\|g\|_{L^2(B)}.
\end{align}
Moreover, notice that $a=L^Mb$, and the function $r_B^{-2M}b$ has similar properties as the atom $a$, that is,
$$\supp\lf(r_B^{-2M}b\r)\subset B\quad {\rm{and}} \quad
\lf\|r_B^{-2M}b\r\|_{L^2(\mathbb{X})}\leq\lf[\mu\lf(B\r)\r]^{\frac1{2}}\lf\|\mathbf{1}_B\r\|^{-1}_{X(\mathbb{X})}.$$
Using these facts, the key estimate \eqref{e5.9}, and similar to the proof of \eqref{e5.6}, we obtain that
\eqref{e5.7} holds true, the details being omitted here.

Now, we deal with the term ${\rm{II_{1}}}$. Let $s_0\in(0,1]$, $q\in(1,\min\{\frac{1}{s_0},\frac{2}{q_0}\}]$,
and $a\in(0,1-\frac{1}{q})$. Then, by the H\"{o}lder inequality, we find that, for any $x\in\mathbb{X}$,
\begin{align*}
&\sum_{i=-\infty}^{i_0-1}\sum_{j\in\nn}\lambda_{i,j}S_{L,\psi}\lf(\lf[I-e^{-r^2_{B_{i,j}}L}\r]^ML^{i\tau}a_{i,j}\r)(x)\\
&\hs\leq\lf(\sum_{i=-\infty}^{i_0-1}2^{iaq'}\r)^{\frac{1}{q'}}
\lf\{\sum_{i=-\infty}^{i_0-1}2^{-iaq}\lf[\sum_{j\in\nn}\lambda_{i,j} S_{L,\psi}\lf(\lf[I-e^{-r^2_{B_{i,j}}L}\r]^ML^{i\tau}a_{i,j}\r)(x)\r]^q\r\}^{\frac{1}{q}}\\
&\hs=\frac{2^{i_0a}}{(2^{aq}-1)^{1/q}}
\lf\{\sum_{i=-\infty}^{i_0-1}2^{-iaq}\lf[\sum_{j\in\nn}\lambda_{i,j} S_{L,\psi}\lf(\lf[I-e^{-r^2_{B_{i,j}}L}\r]^ML^{i\tau}a_{i,j}\r)(x)\r]^q\r\}^{\frac{1}{q}}.
\end{align*}
By this, the assumption that $X^{\frac{1}{s_0}}(\mathbb{X})$ is a BBF space, the well-known inequality, for any $d\in(0,1]$
and $\{\theta_j\}_{j\in\nn}\subset\mathbb{C}$, $(\sum_{j\in\nn}|\theta_j|)^d\leq\sum_{j\in\nn}|\theta_j|^d$,
$q\in(1,\min\{\frac{1}{s_0},\frac{2}{q_0}\}]$, and $\lambda_{i,j}:=2^i\|\mathbf{1}_{B_{i,j}}\|_{X(\mathbb{X})}$, we conclude that
\begin{align}\label{e5.10}
{\rm{II_{1}}}&\lesssim\lf\|\mathbf{1}_{\{x\in\mathbb{X}: \ \frac{2^{i_0a}}{(2^{aq}-1)^{1/q}}
\{\sum_{i=-\infty}^{i_0-1}2^{-iaq}[\sum_{j\in\nn}\lambda_{i,j}S_{L,\psi}([I-e^{-r^2_{B_{i,j}}L}]^M
L^{i\tau}a_{i,j})(x)]^q\}^{\frac{1}{q}}>2^{i_0-2}\}}\r\|_{X(\mathbb{X})}\nonumber\\
&\lesssim2^{-i_0q(1-a)}\lf\|\sum_{i=-\infty}^{i_0-1}2^{-iaq}\lf[\sum_{j\in\nn}
\lambda_{i,j} S_{L,\psi}\lf(\lf[I-e^{-r^2_{B_{i,j}}L}\r]^ML^{i\tau}a_{i,j}\r)\r]^q\r\|_{X(\mathbb{X})}\nonumber\\
&\lesssim2^{-i_0q(1-a)}\lf\|\sum_{i=-\infty}^{i_0-1}2^{-iaqs_0}\sum_{k\in\zz_+}\lf[\sum_{j\in\nn}
\lambda_{i,j} S_{L,\psi}\lf(\lf[I-e^{-r^2_{B_{i,j}}L}\r]^ML^{i\tau}a_{i,j}\r)\mathbf{1}_{U_k(B^\tau_{i,j})}\r]^{qs_0}\r\|
^{\frac{1}{s_0}}_{X^{\frac{1}{s_0}}(\mathbb{X})}\nonumber\\
&\lesssim2^{-i_0q(1-a)}\lf\{\sum_{i=-\infty}^{i_0-1}\sum_{k\in\zz_+}2^{-iqs_0(a-1)}\r.\nonumber\\
&\hs\hs\hs\hs\hs\hs\lf.\times\lf\|\sum_{j\in\nn}\lf[\lf\|\mathbf{1}_{B_{i,j}}\r\|_{X(\mathbb{X})} S_{L,\psi}\lf(\lf[I-e^{-r^2_{B_{i,j}}L}\r]^ML^{i\tau}a_{i,j}\r)
\mathbf{1}_{U_k(B^\tau_{i,j})}\r]^{qs_0}\r\|_{X^{\frac{1}{s_0}}(\mathbb{X})}\r\}^{\frac{1}{s_0}},
\end{align}
where $B^\tau_{i,j}:=(1+|\tau|)B_{i,j}$ for any $i\in\zz$ and $j\in\nn$.

Moreover, by \eqref{e5.6}, we find that, for any $i\in\zz$ and $j\in\nn$,
$$\lf\|\lf[\lf\|\mathbf{1}_{B_{i,j}}\r\|_{X(\mathbb{X})}S_{L,\psi}\lf(\lf[I-e^{-r^2_{B_{i,j}}L}\r]^ML^{i\tau}
a_{i,j}\r)\mathbf{1}_{U_k(B^\tau_{i,j})}\r]^{q}\r\|_{L^{\frac{2}{q}}
(\mathbb{X})}\ls2^{-kq\alpha}\lf[\mu\lf(B_{i,j}\r)\r]^{\frac{q}{2}}.$$
Applying this and Lemma \ref{le2.2} with $q:=\frac{2}{q}$ and $\theta:=2^k(1+|\tau|)$, we conclude that, for any $k\in\zz_+$,
\begin{align}\label{e5.11}
&\lf\|\sum_{j\in\nn}\lf[\lf\|\mathbf{1}_{B_{i,j}}\r\|_{X(\mathbb{X})}S_{L,\psi}\lf(\lf[I-e^{-r^2_{B_{i,j}}L}\r]^ML^{i\tau}
a_{i,j}\r)\mathbf{1}_{U_k(B^\tau_{i,j})}\r]^{qs_0}\r\|_{X^{\frac{1}{s_0}}(\mathbb{X})}\nonumber\\
&\hs\lesssim2^{k(1-\frac{s_0q}{2})n}(1+|\tau|)^{(1-\frac{s_0q}{2})n}\lf\|\sum_{j\in\nn}2^{-kq\alpha s_0} \mathbf{1}_{B_{i,j}}\r\|_{X^{\frac{1}{s_0}}(\mathbb{X})}\nonumber\\
&\hs=(1+|\tau|)^{(1-\frac{s_0q}{2})n}2^{-ks_0[q\alpha-(\frac{1}{s_0}-\frac{q}{2})n]}
\lf\|\sum_{j\in\nn}\mathbf{1}_{B_{i,j}}\r\|_{X^{\frac{1}{s_0}}(\mathbb{X})}.
\end{align}
Since $q>1$ and $\alpha>n(\frac{1}{s_0}-\frac{1}{2})$, it follows that $q>1>\frac{n/s_0}{\alpha+(n/2)}$,
which together with \eqref{e5.10}, \eqref{e5.11}, $a\in(0,1-\frac{1}{q})$, $2^{i_0}\leq\alpha<2^{i_0+1}$,
the assumption that $X^{\frac{1}{s_0}}(\mathbb{X})$
is a BBF space, Remark \ref{re2.1}, \eqref{e5.3}, and similar to the proof of \eqref{e4.16}, further implies that
\begin{align}\label{e5.12}
{\rm{II_{1}}}&\lesssim2^{-i_0q(1-a)}\lf\{\sum_{i=-\infty}^{i_0-1}\sum_{k\in\zz_+}2^{-iqs_0(a-1)}
(1+|\tau|)^{(1-\frac{s_0q}{2})n}2^{-ks_0[q\alpha-(\frac{1}{s_0}-\frac{q}{2})n]}
\lf\|\sum_{j\in\nn}\mathbf{1}_{B_{i,j}}\r\|_{X^{\frac{1}{s_0}}(\mathbb{X})}\r\}^{\frac{1}{s_0}}\nonumber\\
&\lesssim2^{-i_0q(1-a)}(1+|\tau|)^{(\frac{1}{s_0}-\frac{q}{2})n}\lf\{\sum_{i=-\infty}^{i_0-1}2^{-iqs_0(a-1)}
\lf\|\sum_{j\in\nn}\mathbf{1}_{B_{i,j}}\r\|_{X^{\frac{1}{s_0}}(\mathbb{X})}\r\}^{\frac{1}{s_0}}\nonumber\\
&\lesssim2^{-i_0q(1-a)}(1+|\tau|)^{(\frac{1}{s_0}-\frac{q}{2})n}
\lf\{\sum_{i=-\infty}^{i_0-1}2^{-iqs_0(a-1)}2^{-is_0}\r\}^{\frac{1}{s_0}}
\sup_{i\in\zz}2^i\lf\|\sum_{j\in\nn}\mathbf{1}_{B_{i,j}}\r\|^{\frac{1}{s_0}}_{X^{\frac{1}{s_0}}(\mathbb{X})}\nonumber\\
&\lesssim2^{-i_0}(1+|\tau|)^{(\frac{1}{s_0}-\frac{q}{2})n}\sup_{i\in\zz}2^i\lf\|\sum_{j\in\nn}\mathbf{1}_{B_{i,j}}\r\|
^{\frac{1}{s_0}}_{X^{\frac{1}{s_0}}(\mathbb{X})}\nonumber\\
&\lesssim\alpha^{-1}(1+|\tau|)^{(\frac{1}{s_0}-\frac{1}{2})n}\|f\|_{WH_{X,L}(\mathbb{X})}.
\end{align}

Next, we deal with the term ${\rm{II_{2}}}$. Since $\alpha>n(\frac{1}{s_0}-\frac{1}{2})$, it follows that $1>\frac{n/s_0}{\alpha+(n/2)}$. Thus, by choosing $r\in(\frac{n/s_0}{\alpha+(n/2)},1)$,
$\lambda_{i,j}:=2^i\|\mathbf{1}_{B_{i,j}}\|_{X(\mathbb{X})}$, and the assumption that $X^{\frac{1}{s_0}}(\mathbb{X})$
is a BBF space, we know that
\begin{align}\label{e5.13}
{\rm{II_{2}}}&=\lf\|\mathbf{1}_{\{x\in\mathbb{X}: \ \sum_{i=i_0}^\infty\sum_{j\in\nn}\lambda_{i,j}S_{L,\psi}([I-e^{-r^2_{B_{i,j}}L}]^M
L^{i\tau}a_{i,j})(x)>\frac{\alpha}{4}\}}\r\|_{X(\mathbb{X})}\nonumber\\
&\lesssim\alpha^{-r}\lf\|\sum_{i=i_0}^\infty\sum_{j\in\nn}\lf[2^i\lf\|\mathbf{1}_{B_{i,j}}\r\|_{X(\mathbb{X})}
S_{L,\psi}\lf(\lf[I-e^{-r^2_{B_{i,j}}L}\r]^ML^{i\tau}a_{i,j}\r)\r]^r\r\|_{X(\mathbb{X})}\nonumber\\
&\lesssim\alpha^{-r}\lf\|\sum_{i=i_0}^\infty\sum_{j\in\nn}\sum_{k\in\zz_+}\lf[2^i\lf\|\mathbf{1}_{B_{i,j}}\r\|_{X(\mathbb{X})}
S_{L,\psi}\lf(\lf[I-e^{-r^2_{B_{i,j}}L}\r]^ML^{i\tau}a_{i,j}\r)
\mathbf{1}_{U_k(B^\tau_{i,j})}\r]^{rs_0}\r\|^{\frac{1}{s_0}}_{X^{\frac{1}{s_0}}(\mathbb{X})}\nonumber\\
&\lesssim\alpha^{-r}\lf\{\sum_{i=i_0}^\infty\sum_{k\in\zz_+}2^{irs_0}\r.\nonumber\\
&\hs\hs\hs\hs\lf.\times\lf\|\sum_{j\in\nn}\lf[\lf\|\mathbf{1}_{B_{i,j}}\r\|_{X(\mathbb{X})}
S_{L,\psi}\lf(\lf[I-e^{-r^2_{B_{i,j}}L}\r]^ML^{i\tau}a_{i,j}\r)
\mathbf{1}_{U_k(B^\tau_{i,j})}\r]^{rs_0}\r\|_{X^{\frac{1}{s_0}}(\mathbb{X})}\r\}^{\frac{1}{s_0}}.
\end{align}
Moreover, by \eqref{e5.6}, $r\in(0,1)$, we find that, for any $i\in\zz$ and $j\in\nn$,
$$\lf\|\lf[\lf\|\mathbf{1}_{B_{i,j}}\r\|_{X(\mathbb{X})}S_{L,\psi}\lf(\lf[I-e^{-r^2_{B_{i,j}}L}\r]^ML^{i\tau}
a_{i,j}\r)\mathbf{1}_{U_k(B^\tau_{i,j})}\r]^{r}\r\|_{L^{\frac{2}{r}}
(\mathbb{X})}\ls2^{-kr\alpha}\lf[\mu\lf(B_{i,j}\r)\r]^{\frac{r}{2}}.$$
By this, \eqref{e5.13}, \eqref{e5.11} with $q$ there in replaced by $r$,
$r\in(\frac{n/s_0}{\alpha+(n/2)},1)$, the assumption that $X^{\frac{1}{s_0}}(\mathbb{X})$
is a BBF space, Remark \ref{re2.1}, \eqref{e5.3}, and similar to the proof of \eqref{e4.16}, we conclude that
\begin{align}\label{e5.14}
{\rm{II_{2}}}&\lesssim\alpha^{-r}\lf\{\sum_{i=i_0}^\infty\sum_{k\in\zz_+}2^{irs_0}
(1+|\tau|)^{(1-\frac{rs_0}{2})n}2^{-ks_0[r\alpha-(\frac{1}{s_0}-\frac{r}{2})n]}
\lf\|\sum_{j\in\nn}\mathbf{1}_{B_{i,j}}\r\|_{X^{\frac{1}{s_0}}(\mathbb{X})}\r\}^{\frac{1}{s_0}}\nonumber\\
&\lesssim\alpha^{-r}\lf\{\sum_{i=i_0}^\infty2^{irs_0}(1+|\tau|)^{(1-\frac{rs_0}{2})n}
\lf\|\sum_{j\in\nn}\mathbf{1}_{B_{i,j}}\r\|_{X^{\frac{1}{s_0}}(\mathbb{X})}\r\}^{\frac{1}{s_0}}\nonumber\\
&\lesssim\alpha^{-r}(1+|\tau|)^{(\frac1{s_0}-\frac{r}{2})n}\lf\{\sum_{i=i_0}^\infty2^{irs_0}2^{-is_0}\r\}^{\frac{1}{s_0}}
\sup_{i\in\zz}2^i\lf\|\sum_{j\in\nn}\mathbf{1}_{B_{i,j}}\r\|^{\frac{1}{s_0}}_{X^{\frac{1}{s_0}}(\mathbb{X})}\nonumber\\
&\lesssim\alpha^{-r}(1+|\tau|)^{(\frac1{s_0}-\frac{r}{2})n}2^{-i_0(1-r)}
\sup_{i\in\zz}2^i\lf\|\sum_{j\in\nn}\mathbf{1}_{B_{i,j}}\r\|^{\frac{1}{s_0}}_{X^{\frac{1}{s_0}}(\mathbb{X})}\nonumber\\
&\lesssim\alpha^{-1}(1+|\tau|)^{(\frac1{s_0}-\frac{r}{2})n}\|f\|_{WH_{X,L}(\mathbb{X})}.
\end{align}

Finally, we deal with the terms ${\rm{II_{3}}}$ and ${\rm{II_{4}}}$.
Using the key estimate \eqref{e5.7} and arguments similar to the estimates of ${\rm{II_{1}}}$ and ${\rm{II_{2}}}$,
we conclude that
\begin{align}\label{e5.15}
{\rm{II_{3}}}\lesssim\alpha^{-1}(1+|\tau|)^{(\frac1{s_0}-\frac{1}{2})n}\|f\|_{WH_{X,L}(\mathbb{X})}
\end{align}
and
\begin{align}\label{e5.16}
{\rm{II_{4}}}\lesssim\alpha^{-1}(1+|\tau|)^{(\frac1{s_0}-\frac{r}{2})n}\|f\|_{WH_{X,L}(\mathbb{X})},
\end{align}
where $r\in(\frac{n/s_0}{\alpha+(n/2)},1)$, the details being omitted here.
Combining the above estimates of \eqref{e5.16}, \eqref{e5.15}, \eqref{e5.14}, \eqref{e5.12}, and \eqref{e5.5},
we conclude that the claim \eqref{e5.1} holds true.
This completes the proof of Theorem \ref{th1.3}.
\end{proof}




\section{Applications to Specific Function Spaces}\label{s6}
\hskip\parindent
In this section, we apply the results obtained in Sections \ref{s4} and \ref{s5} to some specific spaces including Orlicz spaces, variable Lebesgue spaces, weighted Lebesgue spaces,
and mixed-norm Lebesgue spaces $(\mathbb{X}:=\rn)$.

\subsection{Orlicz Spaces}
\hskip\parindent
Let us recall the definitions of both Orlicz functions and Orlicz spaces (see e.g. \cite{rr91}). A function $\Phi:[0,\fz)\rightarrow[0,\fz)$ is called an Orlicz function if $\Phi$ is non-decreasing, $\Phi(0)=0$, $\Phi(t)>0$ with $t\in(0,\fz)$, and
 $\lim_{t\rightarrow \fz}\Phi(t)=\fz$. For any
$p\in(-\fz,\fz)$, an Orlicz function $\Phi$ is said to be of lower (resp., upper) type $p$ if
there exists a positive constant $C_{(p)}$, depending on $p$, such that, for any $t\in(0,\fz)$
and $s\in(0,1)$ (resp., $s\in[1,\fz)$),
$$\Phi(st)\leq C_{(p)}s^p\Phi(t).$$
In what follows, for any given $s\in(0,\fz)$, let $\Phi_s(t):=\Phi(t^s)$ for any $t\in(0,\fz)$.
Now, we introduce the concept of the Orlicz spaces on $\mathbb{X}$.

\begin{definition}\label{thos}
Let $\Phi$ be an Orlicz function of both lower type $p_{\Phi}^-\in(0,\fz)$ and upper
type $p_{\Phi}^+\in[p_{\Phi}^-,\fz)$. The Orlicz space $L^{\Phi}(X)$ is defined to be the set of all the functions
$f\in\mathfrak{U}(\mathbb{X})$ such that
$$\|f\|_{L^{\Phi}(\mathbb{X})}:=\inf\lf\{\lambda\in(0,\fz):
\int_{\mathbb{X}}\Phi\lf(\frac{|f(x)|}{\lambda}\r)\,d\mu(x)\leq1\r\}<\fz.$$
\end{definition}
\begin{remark}
Notice that $L^{\Phi}(X)$ is a quasi-Banach space, and a BQBF space (see e.g. \cite[Section 8.3]{syy22b}).

\end{remark}

In what follows, we always assume that the Orlicz function $\Phi$ satisfies lower type $p_{\Phi}^-\in(0,\fz)$ and upper
type $p_{\Phi}^+\in[p_{\Phi}^-,\fz)$. When $X(\mathbb{X}):=L^{\Phi}(X)$, then
$WX(\mathbb{X})$ as in Definition \ref{de2.1} is reduced to the weak Orlicz space $WL^{\Phi}(\mathbb{X})$, we denote
$WH_{L,X}(\mathbb{X})$, $WH^M_{X,L,\mathrm{atom}}(\mathbb{X})$, and
$WH^{M,\varepsilon}_{X,L,\mathrm{mol}}(\mathbb{X})$ by $WH^{\Phi}_L(\mathbb{X})$, $WH^{\Phi,M}_{L,\mathrm{atom}}(\mathbb{X})$, and
$WH^{\Phi,M,\varepsilon}_{L,\mathrm{mol}}(\mathbb{X})$, respectively.

\begin{theorem}\label{thos1}
Let the operator $L$ be the same as in Theorem \ref{th1.1} and $\Phi$ be an Orlicz function
with positive lower type $p^-_{\Phi}$ and positive upper type
$p^+_{\Phi}$. Assume that $0<p^-_{\Phi}\leq p^+_{\Phi}<2$,
$M\in(\frac{n}{2}[\frac{1}{\min\{1,r^-_{\Phi}\}}-\frac{1}{2}],\fz)\cap\mathbb{N}$, and
$\varepsilon\in(\frac{n}{\min\{1,r^-_{\Phi}\}},\fz)$.
Then the spaces $WH^{\Phi}_L(\mathbb{X})$, $WH^{\Phi,M}_{L,\mathrm{atom}}(\mathbb{X})$, and
$WH^{\Phi,M,\varepsilon}_{L,\mathrm{mol}}(\mathbb{X})$
coincide with equivalent quasi-norms.
\end{theorem}

\begin{proof}

To prove this theorem, it suffices to show that the Orlicz space $L^{\Phi}$ satisfies Assumption \ref{as1.1} and Assumption \ref{as1.2}.
Choose $p:=p_{\Phi}^-$, $q_0:=2$, and $s_0\in(0,\min\{1,r^-_{\Phi}\})$ such that $M>\frac{n}{2}(\frac{1}{s_0}-\frac{1}{2})$ and $\varepsilon>\frac{n}{s_0}$.
From \cite[Theorem 6.6]{fmy20}, for the above $p$, we know that the BQBF space $X(\mathbb{X})=L^{\Phi}(\mathbb{X})$ satisfies Assumption \ref{as1.1}.

By \cite[p. 61, Proposition 4 and p. 100, Proposition 1]{rr91}, we conclude that,
$$\lf[\lf(\lf[L^{\Phi}(\mathbb{X})\r]^{\frac{1}{s_0}}\r)^{'}\r]^{\frac{1}{(2/s_0)^{'}}}=
L^{\Psi}(\mathbb{X}),$$
where, for any $t\in(0,\fz)$,
$$\Phi(t):=\sup_{u\in(0,\fz)}\lf(t^{\frac{1}{(2/s_0)^{'}}}u-\Phi(u^{\frac{1}{s_0}})\r).$$
Furthermore, from \cite{shyy17}, we know that $\Psi$ is an Orlicz function with positive lower type $p^-_{\Phi}:=\frac{(p_{\Phi}^+/s_0)^{'}}{(2/s_0)^{'}}\in(1,\fz)$.
By the above facts and \cite[Theorem 6.6]{fmy20}, we obtain that the BQBF space $X(\mathbb{X})=L^{\Phi}(\mathbb{X})$ satisfies Assumption \ref{as1.2}.
This completes the proof Theorem \ref{thos1}.
\end{proof}

\begin{theorem}\label{thos2}
Let $\Phi$ be an Orlicz function
with positive lower type $p^-_{\Phi}$ and positive upper type
$p^+_{\Phi}$. Assume that $0<p^-_{\Phi}\leq p^+_{\Phi}<2$.
Assume that $L$ is a non-negative self-adjoint operator on $L^2(\mathbb{X})$
satisfying the Davies-Gaffney estimate \eqref{e1.4}. Let $0<\gamma\neq1$,
$\beta\in[\gamma n(\frac{1}{\min\{1,r^-_{\Phi}\}}-\frac{1}{2}),\infty)$, and $r\in(0,1]$. Then there exists a positive
constant $C$ such that, for any $\tau\in \mathbb{R}$ and $f\in WH^{\Phi}_{L}(\mathbb{X})$,
\begin{align*}
\lf\|(I+L)^{-\beta/2}e^{i\tau L^{\gamma/2}}f\r\|_{WH^{\Phi}_{L}(\mathbb{X})}\leq C\lf(1+|\tau|\r)^{n(\frac{1}{s_0}-\frac{r}{2})}\|f\|_{WH^{\Phi}_{L}(\mathbb{X})}.
\end{align*}
\end{theorem}

\begin{theorem}\label{thos2x}
Let $\Phi$ be an Orlicz function
with positive lower type $p^-_{\Phi}$ and positive upper type
$p^+_{\Phi}$. Assume that $0<p^-_{\Phi}\leq p^+_{\Phi}<2$.
Assume that $L$ is a non-negative self-adjoint operator on $L^2(\mathbb{X})$
satisfying the Davies-Gaffney estimate \eqref{e1.4}. Let $0<\gamma\neq1$,
$\beta\in[\gamma n(\frac{1}{\min\{1,r^-_{\Phi}\}}-\frac{1}{2}),\infty)$, and $r\in(0,1]$. Then there exists a positive
constant $C$ such that, for any $\tau\in \mathbb{R}$ and $f\in H^{\Phi}_{L}(\mathbb{X})$,
\begin{align*}
\lf\|(I+L)^{-\beta/2}e^{i\tau L^{\gamma/2}}f\r\|_{H^{\Phi}_{L}(\mathbb{X})}\leq C\lf(1+|\tau|\r)^{n(\frac{1}{s_0}-\frac{r}{2})}\|f\|_{H^{\Phi}_{L}(\mathbb{X})}.
\end{align*}
\end{theorem}

\begin{theorem}\label{thos3}
Let $(\mathbb{X},d,\mu)$ be a metric measure space satisfying the Ahlfors $n$-regular condition \eqref{e1.6},
and $\Phi$ be an Orlicz function
with positive lower type $p^-_{\Phi}$ and positive upper type
$p^+_{\Phi}$. Assume that $0<p^-_{\Phi}\leq p^+_{\Phi}<2$.
Assume that $L$ is a non-negative self-adjoint operator on $L^2(\mathbb{X})$
satisfying the Gaussian upper estimate \eqref{e1.7} and \eqref{e1.8}.
Let $\alpha>n(\frac{1}{\min\{1,p_{\Phi}^-\}}-\frac{1}{2})$ and $r\in(\frac{n/\min\{1,p_{\Phi}^-\}}{\alpha+n/2},1]$. Then there exists a positive constant $C$ such that, for any $\tau\in \mathbb{R}$ and $f\in WH^{\Phi}_{L}(\mathbb{X})$,
\begin{align*}
\lf\|L^{i\tau}f\r\|_{WH^{\Phi}_{L}(\mathbb{X})}\leq C\lf(1+|\tau|\r)^{n(\frac{1}{s_0}-\frac{r}{2})}\|f\|_{WH^{\Phi}_{L}(\mathbb{X})}.
\end{align*}
\end{theorem}

\begin{theorem}\label{thos3x}
Let $(\mathbb{X},d,\mu)$ be a metric measure space satisfying the Ahlfors $n$-regular condition \eqref{e1.6},
and $\Phi$ be an Orlicz function
with positive lower type $p^-_{\Phi}$ and positive upper type
$p^+_{\Phi}$. Assume that $0<p^-_{\Phi}\leq p^+_{\Phi}<2$.
Assume that $L$ is a non-negative self-adjoint operator on $L^2(\mathbb{X})$
satisfying the Gaussian upper estimate \eqref{e1.7} and \eqref{e1.8}.
Let $\alpha>n(\frac{1}{\min\{1,p_{\Phi}^-\}}-\frac{1}{2})$ and $r\in(\frac{n/\min\{1,p_{\Phi}^-\}}{\alpha+n/2},1]$. Then there exists a positive constant $C$ such that, for any $\tau\in \mathbb{R}$ and $f\in H^{\Phi}_{L}(\mathbb{X})$,
\begin{align*}
\lf\|L^{i\tau}f\r\|_{H^{\Phi}_{L}(\mathbb{X})}\leq C\lf(1+|\tau|\r)^{n(\frac{1}{s_0}-\frac{r}{2})}\|f\|_{H^{\Phi}_{L}(\mathbb{X})}.
\end{align*}
\end{theorem}

The proofs of Theorem \ref{thos2}, \ref{thos2x}, \ref{thos3} and \ref{thos3x} are similar to that of Theorem \ref{thos1}, we omit the details here.

\begin{remark}
 To the best of our knowledge, when it comes to the Euclidean spaces setting $(\mathbb{X}:=\rn)$, Theorems \ref{thos2}, \ref{thos2x}, \ref{thos3} and \ref{thos3x} are also completely
new.
\end{remark}

\subsection{Variable Lebesgue Spaces }
\hskip\parindent
In this subsection, we recall the definition of variable Lebesgue spaces and refer the readers to \cite{cf13,dhh11}
for more details. In what following, we say that a $\mu$-measurable function $p(\cdot):\mathbb{X}\rightarrow(0,\fz]$ is a variable
exponent. For any variable exponent $p(\cdot)$, let $\tilde{p}^-$ denote its essential infimum and $\tilde{p}^+$
its essential supremum.
\begin{definition} Let $p(\cdot)$ be a variable exponent with $0<\tilde{p}^-\leq \tilde{p}^+\leq\fz$, and
$$\mathbb{X}_{\fz}:=\{x\in\mathbb{X}:p(x)=\fz\}.$$
The variable Lebesgue space $L^{p(\cdot)}(\mathbb{X})$ is defined to be the set of all the functions
$f\in\mathfrak{U}(\mathbb{X})$ such that
$$
\|f\|_{L^{p(\cdot)}(\mathbb{X})}:=\inf\lf\{\lambda\in(0,\fz):\rho_{p(\cdot)}
(\frac{f}{\lambda})\leq1\r\}<\fz,$$
where, for any $f\in\mathfrak{U}(X)$,
$$\rho_{p(\cdot)}(f):=\int_{\mathbb{X}\backslash \mathbb{X}_{\fz}}
|f(x)|^{p(x)}\,d\mu(x)+\|f\|_{L^{\fz}(\mathbb{X}_{\fz})}.$$
\end{definition}

Let $p(\cdot)$ be as the above definition. Observe that $L^{p(\cdot)}(\mathbb{X})$
 is a quasi-Banach function
space and hence a BQBF space (see e.g. \cite[Remark 2.7(iv)]{yhyy23}). If
$X(\mathbb{X}) =L^{p(\cdot)}(\mathbb{X})$, then $WX(\mathbb{X})$ as in Definition \ref{de2.1} becomes the variable weak Lebesgue
space $WL^{p(\cdot)}(\mathbb{X})$.
Furthermore, we denote
$WH_{L,X}(\mathbb{X})$, $WH^M_{X,L,\mathrm{atom}}(\mathbb{X})$, and
$WH^{M,\varepsilon}_{X,L,\mathrm{mol}}(\mathbb{X})$ by $WH^{p(\cdot)}_L(\mathbb{X})$, $WH^{p(\cdot),M}_{L,\mathrm{atom}}(\mathbb{X})$, and
$WH^{p(\cdot),M,\varepsilon}_{L,\mathrm{mol}}(\mathbb{X})$, respectively.

Let $x_0\in\mathbb{X}$ be the fixed point. Recall that a variable exponent $p(\cdot)$, with $p^+<\fz$, is said to be globally log-H\"{o}lder continuous if there exist constants
$C\in(0,\fz)$
and $p_{\fz}\in\mathbb{R}$ such that, for any $x, y\in\mathbb{X}$,
$$|p(x)-p(y)|\leq C\frac{1}{\log(e+\frac{1}{d(x,\,y)})}$$
and
$$|p(x)-p_{\fz}|\leq C\frac{1}{\log(e+d(x,\,x_0))}.$$

\begin{theorem}\label{thvs1}
 Let $L$ be a non-negative self-adjoint operator on $L^2(\mathbb{X})$
satisfying the Davies-Gaffney estimate \eqref{e1.4}, and $p(\cdot):\mathbb{X}\rightarrow(0,\fz)$ be
log-H\"{o}lder continuous. Assume that $0<\tilde{p}^-\leq \tilde{p}^+<2$,
$M \in(\frac n2[\frac{1}{\min\{1,\tilde{p}^-\}}-\frac12],\fz)\cap\mathbb{N}$,
and $\varepsilon\in(\frac{n}{\min\{1,\tilde{p}^-\}},\fz)$. The the spaces
$WH^{p(\cdot)}_L(\mathbb{X})$, $WH^{p(\cdot),M}_{L,\mathrm{atom}}(\mathbb{X})$, and
$WH^{p(\cdot),M,\varepsilon}_{L,\mathrm{mol}}(\mathbb{X})$
coincide with equivalent quasi-norms.
\end{theorem}
\begin{proof}
Take $p:=\tilde{p}^-$, $q_0:=2$, and $s_0\in(0,\,\min\{1,\tilde{p}^-\})$ such that $M \in(\frac n2[\frac{1}{\min\{1,\tilde{p}^-\}}-\frac12],\fz)\cap\mathbb{N}$ and $\varepsilon\in(\frac{n}{s_0},\fz)$.
From \cite[Theorem 2.7]{zsy16}, we know that the BQBF space $X(\mathbb{X})=L^{p(\cdot)}(\mathbb{X})$ satisfies Assumption \ref{as1.1} for above exponents.
Moreover, by \cite[Lemma 2.9]{zsy16}, we obtain that
$$\lf[\lf(\lf[L^{p(\cdot)}(\mathbb{X})\r]^{\frac{1}{s_0}}\r)^{'}\r]^{\frac{1}{(2/s_0)^{'}}}=
L^{\frac{(p(\cdot)/s_0)^{'}}{(2/s_0)^{'}}}(\mathbb{X}),$$
where, for any $x\in\mathbb{X}$, $\frac{1}{(p(x)/s_0)^{'}}+\frac{1}{p(x)/s_0}=1$.
Therefore, by the assumption that $0<\tilde{p}^-\leq \tilde{p}^+<2$ and \cite[Lemma 2.5]{zsy16}, we obtain that the BQBF space $X(\mathbb{X})=L^{p(\cdot)}(\mathbb{X})$ satisfies Assumption \ref{as1.2} for above exponents. Therefore, all the assumption of Theorem \ref{th1.1} are satisfied, which implies the desired result. This completes the proof of Theorem \ref{thvs1}.
\end{proof}

\begin{theorem}\label{thvs2}
 Let $L$ be a non-negative self-adjoint operator on $L^2(\mathbb{X})$
satisfying the Davies-Gaffney estimate \eqref{e1.4}, and $p(\cdot):\mathbb{X}\rightarrow(0,\fz)$ be
log-H\"{o}lder continuous. Assume that $0<\tilde{p}^-\leq \tilde{p}^+<2$, $0<\gamma\neq1$,
$\beta\in[\gamma n(\frac{1}{\min\{1,\tilde{p}^-\}}-\frac{1}{2}),\infty)$ and $r\in(0,1]$. Then there exists a positive
constant $C$ such that, for any $\tau\in \mathbb{R}$ and $f\in WH^{p(\cdot)}_{L}(\mathbb{X})$,
\begin{align*}
\lf\|(I+L)^{-\beta/2}e^{i\tau L^{\gamma/2}}f\r\|_{WH^{p(\cdot)}_{L}(\mathbb{X})}\leq C\lf(1+|\tau|\r)^{n(\frac{1}{\min\{1,\tilde{p}^-\}}-\frac{r}{2})}\|f\|_{WH^{p(\cdot)}_{L}(\mathbb{X})}.
\end{align*}
\end{theorem}

\begin{theorem}\label{thvs2x}
 Let $L$ be a non-negative self-adjoint operator on $L^2(\mathbb{X})$
satisfying the Davies-Gaffney estimate \eqref{e1.4}, and $p(\cdot):\mathbb{X}\rightarrow(0,\fz)$ be
log-H\"{o}lder continuous. Assume that $0<\tilde{p}^-\leq \tilde{p}^+<2$, $0<\gamma\neq1$,
$\beta\in[\gamma n(\frac{1}{\min\{1,\tilde{p}^-\}}-\frac{1}{2}),\infty)$ and $r\in(0,1]$. Then there exists a positive
constant $C$ such that, for any $\tau\in \mathbb{R}$ and $f\in H^{p(\cdot)}_{L}(\mathbb{X})$,
\begin{align*}
\lf\|(I+L)^{-\beta/2}e^{i\tau L^{\gamma/2}}f\r\|_{H^{p(\cdot)}_{L}(\mathbb{X})}\leq C\lf(1+|\tau|\r)^{n(\frac{1}{\min\{1,\tilde{p}^-\}}-\frac{r}{2})}\|f\|_{H^{p(\cdot)}_{L}(\mathbb{X})}.
\end{align*}
\end{theorem}

\begin{theorem}\label{thvs3}
Let $(\mathbb{X},d,\mu)$ be a metric measure space satisfying the Ahlfors $n$-regular condition \eqref{e1.6}.
Let $L$ be a non-negative self-adjoint operator on $L^2(\mathbb{X})$
satisfying the Gaussian upper estimate \eqref{e1.7} and \eqref{e1.8},
and $p(\cdot):\mathbb{X}\rightarrow(0,\fz)$ be
log-H\"{o}lder continuous. Assume that $0<\tilde{p}^-\leq \tilde{p}^+<2$,
 $\alpha>n(\frac{1}{\min\{1,\tilde{p}^-\}}-\frac{1}{2})$ and $r\in(\frac{n/\min\{1,\tilde{p}^-\}}{\alpha+n/2},1]$. Then there exists a positive constant $C$ such that, for any $\tau\in \mathbb{R}$ and $f\in WH^{p(\cdot)}_{L}(\mathbb{X})$,
\begin{align*}
\lf\|L^{i\tau}f\r\|_{WH^{p(\cdot)}_{L}(\mathbb{X})}\leq C\lf(1+|\tau|\r)^{n(\frac{1}{\min\{1,\tilde{p}^-\}}-\frac{r}{2})}\|f\|_{WH^{p(\cdot)}_{L}(\mathbb{X})}.
\end{align*}
\end{theorem}

\begin{theorem}\label{thvs3x}
Let $(\mathbb{X},d,\mu)$ be a metric measure space satisfying the Ahlfors $n$-regular condition \eqref{e1.6}.
Let $L$ be a non-negative self-adjoint operator on $L^2(\mathbb{X})$
satisfying the Gaussian upper estimate \eqref{e1.7} and \eqref{e1.8},
and $p(\cdot):\mathbb{X}\rightarrow(0,\fz)$ be
log-H\"{o}lder continuous. Assume that $0<\tilde{p}^-\leq \tilde{p}^+<2$,
 $\alpha>n(\frac{1}{\min\{1,\tilde{p}^-\}}-\frac{1}{2})$ and $r\in(\frac{n/\min\{1,\tilde{p}^-\}}{\alpha+n/2},1]$. Then there exists a positive constant $C$ such that, for any $\tau\in \mathbb{R}$ and $f\in H^{p(\cdot)}_{L}(\mathbb{X})$,
\begin{align*}
\lf\|L^{i\tau}f\r\|_{H^{p(\cdot)}_{L}(\mathbb{X})}\leq C\lf(1+|\tau|\r)^{n(\frac{1}{\min\{1,\tilde{p}^-\}}-\frac{r}{2})}\|f\|_{H^{p(\cdot)}_{L}(\mathbb{X})}.
\end{align*}
\end{theorem}

The proofs of Theorem \ref{thvs2}, \ref{thvs2x}, \ref{thvs3} and \ref{thvs3x} are similar to that of Theorem \ref{thvs1}, we omit the details here.
\begin{remark}
 We should point out that, when it comes to the Euclidean spaces setting $(\mathbb{X}:=\rn)$, Theorems \ref{thvs2}, \ref{thvs2x}, \ref{thvs3} and \ref{thvs3x} are also completely 
new.
\end{remark}

\subsection{Weighted Lebesgue Spaces}
\hskip\parindent
In this subsection, we recall the concept of weights on $\mathbb{X}$ (see e.g.
\cite[Chapter I]{s89}). Let $\mathfrak{B}$ denote the set of all balls of $\mathbb{X}$. A locally integrable
function $w:\mathbb{X}\rightarrow[0,\fz)$ is called an $\mathbf{A}_p(\mathbb{X})$-weight with
$p\in[1,\fz)$ if
$$[w]_{\mathbf{A}_p(\mathbb{X})}:=\sup_{B\in\mathfrak{B}}	
\mu(B)^{-p}\|w\|_{L^1(B)}
\lf\|w^{-1}\r\|_{L^{\frac{1}{p-1}}(B)}<\fz,$$
where $1/(p-1):=\fz$ when $p=1$. Let
$$\mathbf{A}_{\fz}(\mathbb{X}):=\bigcup_{p\in[1,\fz)}\mathbf{A}_p(\mathbb{X}).$$
For any $\omega\in\mathbf{A}_{\fz}(\mathbb{X})$, the critical index $q_{w}$ of $w$ is defined by setting
\begin{align}\label{qw}
q_{w}:=\inf\lf\{p\in[1,\fz): w\in\mathbf{A}_p(\mathbb{X})\r\}.
\end{align}
Define the space:
\begin{align}\label{www}
\|f\|_{L^p_w(\mathbb{X})}:=\lf\{f\in\mathfrak{U}(\mathbb{X})
:\|f\|_{L^p_w(\mathbb{X})}:\lf(\int_{\mathbb{X}}
|f(x)|^pw(x)\,d\mu(x)\r)^{1/p}<\fz\r\}.
\end{align}
Let $p\in(0,\fz)$ and $w\in\mathbf{A}_{\fz}(\mathbb{X})$. The weighted Lebesgue space
$L_{w}^p(\mathbb{X})$ is defined as
in \eqref{www}. Observe that $L_{w}^p(\mathbb{X})$ is a BQBF space (see e.g. \cite[Section 7.1]{shyy17}),
but not necessarily a quasi-Banach function space (see e.g. \cite[Remark
5.22(ii)]{wyyz21} on the Euclidean space case). If $X(\mathbb{X})=L_{w}^p(\mathbb{X})$, then
$WX(\mathbb{X})$ as in Definition \ref{de2.1} becomes the weighted weak Lebesgue space
$WL_{w}^p(\mathbb{X})$.
Furthermore, we denote
$WH_{L,X}(\mathbb{X})$, $WH^M_{X,L,\mathrm{atom}}(\mathbb{X})$, and
$WH^{M,\varepsilon}_{X,L,\mathrm{mol}}(\mathbb{X})$ by $WH^{p}_{w,L}(\mathbb{X})$, $WH^{p,M}_{w,L,\mathrm{atom}}(\mathbb{X})$, and
$WH^{p,M,\varepsilon}_{w,L,\mathrm{mol}}(\mathbb{X})$, respectively.
 The following
theorem is an application of the results obtained in the above sections to weighted
Lebesgue spaces.

\begin{theorem}\label{thws1}
Let $L$ be a non-negative self-adjoint operator on $L^2(\mathbb{X})$
satisfying the Davies-Gaffney estimate \eqref{e1.4}, $w\in \mathbf{A}_{\fz}(\mathbb{X})$ and $\tilde{p}\in(0,2)$, $s_0\in(0,\min\{1,\tilde{p}/q_w\})$ and $w^{1-(\tilde{p}/s_0)^{'}}(\mathbb{X})
\in\mathbf{A}_{(\tilde{p}/s_0)^{'}/(2/s_0)^{'}}$ with $q_w$ as in \eqref{qw}.  Then the spaces
$WH^{\tilde{p}}_{w,L}(\mathbb{X})$, $WH^{\tilde{p},M}_{w,L,\mathrm{atom}}(\mathbb{X})$, and
$WH^{\tilde{p},M,\varepsilon}_{w,L,\mathrm{mol}}(\mathbb{X})$
coincide with equivalent quasi-norms.
\end{theorem}

\begin{proof}[Proof of Theorem \ref{thws1}]
In order to prove this theorem, we only need to show that the weighted Lebesgue space $L_w^{\tilde{p}}(\mathbb{X})$ satisfies all the
assumptions needed, respectively, in Theorems \ref{th1.1}, \ref{th1.2} and \ref{th1.3}, which
then further implies the desired conclusions of the present theorem.

Take $p:=\tilde{p}$, $q_0:=2$, and $s_0\in(0,\,\min\{1,\tilde{p}/q_w\})$.
From \cite[Theorem 6.5(ii)]{fmy20}, we know that the BQBF space $X(\mathbb{X})=L^{p}_w(\mathbb{X})$ satisfies Assumption \ref{as1.1} for above exponents.
Moreover, we see that
$$\lf[\lf(\lf[L^{\tilde{p}}_w(\mathbb{X})\r]^{\frac{1}{s_0}}\r)^{'}\r]^{\frac{1}{(2/s_0)^{'}}}=
L_{w^{1-(\tilde{p}/s_0)^{'}}}^{\frac{(\tilde{p}/s_0)^{'}}{(2/s_0)^{'}}}(\mathbb{X}).$$
By this and the assumption that $w^{1-(\tilde{p}/s_0)^{'}}(\mathbb{X})
\in\mathbf{A}_{(\tilde{p}/s_0)^{'}/(2/s_0)^{'}}(\mathbb{X})$ with $q_w$ as in \eqref{qw}, we obtain that the BQBF space $X(\mathbb{X})=L_w^{\tilde{p}}(\mathbb{X})$ satisfies Assumption \ref{as1.2} for above exponents. Therefore, the space $L_w^{\tilde{p}}(\mathbb{X})$ satisfies all the assumption of Theorem \ref{th1.1}.  This completes the proof of Theorem \ref{thvs1}.
\end{proof}

\begin{theorem}\label{thws2}
Let $L$ be a non-negative self-adjoint operator on $L^2(\mathbb{X})$
satisfying the Davies-Gaffney estimate \eqref{e1.4}, $w\in \mathbf{A}_{\fz}(\mathbb{X})$ and $\tilde{p}\in(0,2)$, $s_0\in(0,\min\{1,\tilde{p}/q_w\})$ and $w^{1-(\tilde{p}/s_0)^{'}}(\mathbb{X})
\in\mathbf{A}_{(\tilde{p}/s_0)^{'}/(2/s_0)^{'}}$ with $q_w$ as in \eqref{qw}. Assume that  $0<\gamma\neq1$,
$\beta\in[\gamma n(\frac{1}{s_0}-\frac{1}{2}),\infty)$, and $r\in(0,1]$. Then there exists a positive
constant $C$ such that, for any $\tau\in \mathbb{R}$ and $f\in WH^p_{w,L}(\mathbb{X})$,
\begin{align*}
\lf\|(I+L)^{-\beta/2}e^{i\tau L^{\gamma/2}}f\r\|_{WH^p_{w,L}(\mathbb{X})}\leq C\lf(1+|\tau|\r)^{n(\frac{1}{s_0}-\frac{r}{2})}\|f\|_{WH^p_{w,L}(\mathbb{X})}.
\end{align*}
\end{theorem}

\begin{theorem}\label{thws2x}
Let $L$ be a non-negative self-adjoint operator on $L^2(\mathbb{X})$
satisfying the Davies-Gaffney estimate \eqref{e1.4}, $w\in \mathbf{A}_{\fz}(\mathbb{X})$ and $\tilde{p}\in(0,2)$, $s_0\in(0,\min\{1,\tilde{p}/q_w\})$ and $w^{1-(\tilde{p}/s_0)^{'}}(\mathbb{X})
\in\mathbf{A}_{(\tilde{p}/s_0)^{'}/(2/s_0)^{'}}$ with $q_w$ as in \eqref{qw}. Assume that  $0<\gamma\neq1$,
$\beta\in[\gamma n(\frac{1}{s_0}-\frac{1}{2}),\infty)$, and $r\in(0,1]$. Then there exists a positive
constant $C$ such that, for any $\tau\in \mathbb{R}$ and $f\in H^p_{w,L}(\mathbb{X})$,
\begin{align*}
\lf\|(I+L)^{-\beta/2}e^{i\tau L^{\gamma/2}}f\r\|_{H^p_{w,L}(\mathbb{X})}\leq C\lf(1+|\tau|\r)^{n(\frac{1}{s_0}-\frac{r}{2})}\|f\|_{H^p_{w,L}(\mathbb{X})}.
\end{align*}
\end{theorem}

\begin{theorem}\label{thws3}
Let $L$ be a non-negative self-adjoint operator on $L^2(\mathbb{X})$
satisfying the Davies-Gaffney estimate \eqref{e1.4}, $w\in \mathbf{A}_{\fz}(\mathbb{X})$ and $\tilde{p}\in(0,2)$, $s_0\in(0,\min\{1,\tilde{p}/q_w\})$ and $w^{1-(\tilde{p}/s_0)^{'}}(\mathbb{X})
\in\mathbf{A}_{(\tilde{p}/s_0)^{'}/(2/s_0)^{'}}(\mathbb{X})$ with $q_w$ as in \eqref{qw}.
Assume that  $\alpha>n(\frac{1}{s_0}-\frac{1}{2})$ and $r\in(\frac{n/s_0}{\alpha+n/2},1]$. Then there exists a positive constant $C$ such that, for any $\tau\in \mathbb{R}$ and $f\in WH^p_{w,L}(\mathbb{X})$,
\begin{align*}
\lf\|L^{i\tau}f\r\|_{WH^p_{w,L}(\mathbb{X})}\leq C\lf(1+|\tau|\r)^{n(\frac{1}{s_0}-\frac{r}{2})}\|f\|_{WH^p_{w,L}(\mathbb{X})}.
\end{align*}
\end{theorem}

\begin{theorem}\label{thws3x}
Let $L$ be a non-negative self-adjoint operator on $L^2(\mathbb{X})$
satisfying the Davies-Gaffney estimate \eqref{e1.4}, $w\in \mathbf{A}_{\fz}(\mathbb{X})$ and $\tilde{p}\in(0,2)$, $s_0\in(0,\min\{1,\tilde{p}/q_w\})$ and $w^{1-(\tilde{p}/s_0)^{'}}(\mathbb{X})
\in\mathbf{A}_{(\tilde{p}/s_0)^{'}/(2/s_0)^{'}}(\mathbb{X})$ with $q_w$ as in \eqref{qw}.
Assume that  $\alpha>n(\frac{1}{s_0}-\frac{1}{2})$ and $r\in(\frac{n/s_0}{\alpha+n/2},1]$. Then there exists a positive constant $C$ such that, for any $\tau\in \mathbb{R}$ and $f\in H^p_{w,L}(\mathbb{X})$,
\begin{align*}
\lf\|L^{i\tau}f\r\|_{H^p_{w,L}(\mathbb{X})}\leq C\lf(1+|\tau|\r)^{n(\frac{1}{s_0}-\frac{r}{2})}\|f\|_{H^p_{w,L}(\mathbb{X})}.
\end{align*}
\end{theorem}

The proofs of Theorems \ref{thws2}, \ref{thws2x}, \ref{thws3} and \ref{thws3x} are similar to that of Theorem \ref{thws1}, we omit the details here.

\begin{remark}
 To the best of our knowledge, when it comes to the Euclidean spaces setting $(\mathbb{X}:=\rn)$, Theorems \ref{thws2}, \ref{thws2x}, \ref{thws3} and \ref{thws3x} are also
new.
\end{remark}

\subsection{Mixed-Norm Lebesgue Spaces on $\mathbb{R}^n$}
\hskip\parindent
In this subsection, we apply ball quasi-Banach spaces to mixed-norm Lebesgue spaces when $\mathbb{X}:=\rn$. The mixed-norm Lebesgue space $L^{\vec{p}}(\rn)$ was studied by Benedek and Panzone \cite{bp61} in 1961,
which can be traced back to H\"{o}rmander \cite{h60}. Particularly, in order to meet the requirements arising in the study of the
boundedness of operators, {\bf PDE}s, and some other analysis fields, the real-variable theory of mixed-norm function spaces has rapidly been developed in recent years (see e.g. \cite{hy21}).
\begin{definition}
Let $\vec{p}:=(p_1,\ldots,p_n)\in(0,\fz]^n$. The mixed-norm Lebesgue space $L^{\vec{p}}(\rn)$
 is defined to be the set of all the measurable functions $f$ on $\rn$
such that
$$\|f\|_{L^{\vec{p}}(\rn)}:=\lf\{\int_{\rr}\cdots\lf[\int_{\rr}\lf|f(x_1,\ldots,x_n)\r|^{p_1}\,dx_1\r]
^{\frac{p_2}{p_1}}\cdots\,dx_n\r\}^{\frac{1}{p_n}}<\fz$$
with the usual modifications made when $p_i=\fz$ for some $i\in\{1, \ldots,n\}$. Moreover, let
\begin{align}\label{em.1}
 \tilde{p}^-:=\min\lf\{p_1,\ldots,p_n\r\}  \ \ \ \mathrm{and} \ \ \
 \tilde{p}^+:=\max\lf\{p_1,\ldots,p_n\r\}.
 \end{align}
\end{definition}
\begin{remark}
 Let $\vec{p}:=(p_1,\ldots,p_n)\in(0,\fz]^n$
and both $\tilde{p}^-$ and $\tilde{p}^+$ be the same as in \eqref{em.1}. By
Definition \ref{de2.1}, we can conclude that $L^{\vec{p}}(\rn)$ is a ball quasi-Banach space, but it is worth pointing
out that $L^{\vec{p}}(\rn)$ may not be a quasi-Banach function space (see e.g. \cite[Remark 7.21]{zyyw21}).
\end{remark}
If $X(\mathbb{X})=L^{\vec{p}}(\rn)$, then
$WX(\mathbb{X})$ as in Definition \ref{de2.1} becomes the weak mixed-norm Lebesgue space
$WL^{\vec{p}}(\rn)$.
Furthermore, we denote
$WH_{L,X}(\mathbb{X})$, $WH^M_{X,L,\mathrm{atom}}(\mathbb{X})$, and
$WH^{M,\varepsilon}_{X,L,\mathrm{mol}}(\mathbb{X})$ by $WH^{\vec{p}}_{L}(\rn)$, $WH^{\vec{p},M}_{L,\mathrm{atom}}(\rn)$, and
$WH^{\vec{p},M,\varepsilon}_{L,\mathrm{mol}}(\rn)$, respectively.
 The following
theorem is an application of the results obtained in the above sections to mixed-norm
Lebesgue spaces.
\begin{theorem}\label{thms1}
 Let $\vec{p}:=(p_1,\ldots,p_n)\in(0,\fz]^n$, $\tilde{p}^-$ and $\tilde{p}^+$ be the same as in \eqref{em.1},
 and $L$ is a non-negative self-adjoint operator on $L^2(\rn)$
satisfying the Davies-Gaffney estimate \eqref{e1.4}.
 Assume that $0<\tilde{p}^-\leq \tilde{p}^+<2$,
$M \in(\frac n2[\frac{1}{\min\{1,\tilde{p}^-}\}-\frac12],\fz)\cap\mathbb{N}$,
and $\varepsilon\in(\frac{n}{\min\{1,\tilde{p}^-\}},\fz)$. Then the spaces
$WH^{\vec{p}}_L(\rn)$, $WH^{\vec{p},M}_{L,\mathrm{atom}}(\rn)$, and
$WH^{\vec{p},M,\varepsilon}_{L,\mathrm{mol}}(\rn)$
coincide with equivalent quasi-norms.
\end{theorem}
\begin{proof}
Take $p:=\tilde{p}^-$, $q_0:=2$, and $s_0\in(0,\,\min\{1,\tilde{p}^-\})$ such that $M \in(\frac n2[\frac{1}{\min\{1,\tilde{p}^-\}}-\frac12],\fz)\cap\mathbb{N}$ and $\varepsilon\in(\frac{n}{s_0},\fz)$.
From \cite[Theorem 2.7]{zsy16}, we know that the BQBF space $X(\mathbb{X})=L^{\vec{p}}(\rn)$ satisfies Assumption \ref{as1.1} for above exponents.
Moreover, by \cite[Lemma 2.9]{zsy16}, we obtain that
$$\lf[\lf(\lf[L^{\vec{p}}(\rn)\r]^{\frac{1}{s_0}}\r)^{'}\r]^{\frac{1}{(2/s_0)^{'}}}=
L^{\frac{(\vec{p}/s_0)^{'}}{(2/s_0)^{'}}}(\rn),$$
where, for any $i\in[1,n]\cap\nn$, $\frac{1}{(p_i/s_0)^{'}}+\frac{1}{p_i/s_0}=1$.
Therefore, by the assumption that $0<\tilde{p}^-\leq \tilde{p}^+<2$ and \cite[Lemma 2.5]{zsy16}, we obtain that the BQBF space $X(\mathbb{X})=L^{\vec{p}}(\rn)$ satisfies Assumption \ref{as1.2} for above exponents. Therefore, the space $L^{\vec{p}}(\mathbb{X})$ satisfies all the assumption of Theorem \ref{th1.1}, which implies the desired result. This completes the proof of Theorem \ref{thvs1}.
\end{proof}

\begin{theorem}\label{thms2}
Let $\vec{p}:=(p_1,\ldots,p_n)\in(0,\fz]^n$
and both $\tilde{p}^-$ and $\tilde{p}^+$ be the same as in \eqref{em.1}, and $L$ is a non-negative self-adjoint operator on $L^2(\rn)$
satisfying the Davies-Gaffney estimate \eqref{e1.4}.
Assume that $0<\tilde{p}^-\leq \tilde{p}^+<2$, $0<\gamma\neq1$,
$\beta\in[\gamma n(\frac{1}{\min\{1,\tilde{p}^-\}}-\frac{1}{2}),\infty)$, and $r\in(0,1]$. Then there exists a positive
constant $C$ such that, for any $\tau\in \mathbb{R}$ and $f\in WH^{\vec{p}}_{L}(\rn)$,
\begin{align*}
\lf\|(I+L)^{-\beta/2}e^{i\tau L^{\gamma/2}}f\r\|_{WH^{\vec{p}}_{L}(\rn)}\leq C\lf(1+|\tau|\r)^{^{n(\frac{1}{\min\{1,\tilde{p}^-\}}-\frac{r}{2})}}\|f\|_{WH^{\vec{p}}_{L}(\rn)}.
\end{align*}
\end{theorem}

\begin{theorem}\label{thms2x}
Let $\vec{p}:=(p_1,\ldots,p_n)\in(0,\fz]^n$
and both $\tilde{p}^-$ and $\tilde{p}^+$ be the same as in \eqref{em.1}, and $L$ is a non-negative self-adjoint operator on $L^2(\rn)$
satisfying the Davies-Gaffney estimate \eqref{e1.4}.
Assume that $0<\tilde{p}^-\leq \tilde{p}^+<2$, $0<\gamma\neq1$,
$\beta\in[\gamma n(\frac{1}{\min\{1,\tilde{p}^-\}}-\frac{1}{2}),\infty)$, and $r\in(0,1]$. Then there exists a positive
constant $C$ such that, for any $\tau\in \mathbb{R}$ and $f\in H^{\vec{p}}_{L}(\rn)$,
\begin{align*}
\lf\|(I+L)^{-\beta/2}e^{i\tau L^{\gamma/2}}f\r\|_{H^{\vec{p}}_{L}(\rn)}\leq C\lf(1+|\tau|\r)^{^{n(\frac{1}{\min\{1,\tilde{p}^-\}}-\frac{r}{2})}}\|f\|_{H^{\vec{p}}_{L}(\rn)}.
\end{align*}
\end{theorem}

\begin{theorem}\label{thms3}
Let $\vec{p}:=(p_1,\ldots,p_n)\in(0,\fz]^n$
and both $\tilde{p}^-$ and $\tilde{p}^+$ be the same as in \eqref{em.1}, and $L$ be a non-negative self-adjoint operator on $L^2(\rn)$
satisfying the Gaussian upper estimate \eqref{e1.7} and \eqref{e1.8}.
Assume that $0<\tilde{p}^-\leq \tilde{p}^+<2$, $\alpha>n(\frac{1}{\min\{1,\tilde{p}^-\}}-\frac{1}{2})$ and $r\in(\frac{n/\min\{1,\tilde{p}^-\}}{\alpha+n/2},1]$. Then there exists a positive constant $C$ such that, for any $\tau\in \mathbb{R}$ and $f\in WH^{\vec{p}}_{L}(\rn)$,
\begin{align}\label{e1.9}
\lf\|L^{i\tau}f\r\|_{WH^{\vec{p}}_{L}(\rn)}\leq C\lf(1+|\tau|\r)^{n(\frac{1}{\min\{1,\tilde{p}^-\}}-\frac{r}{2})}\|f\|_{WH^{\vec{p}}_{L}(\rn)}.
\end{align}
\end{theorem}

\begin{theorem}\label{thms3x}
Let $\vec{p}:=(p_1,\ldots,p_n)\in(0,\fz]^n$
and both $\tilde{p}^-$ and $\tilde{p}^+$ be the same as in \eqref{em.1}, and $L$ be a non-negative self-adjoint operator on $L^2(\rn)$
satisfying the Gaussian upper estimate \eqref{e1.7} and \eqref{e1.8}.
Assume that $0<\tilde{p}^-\leq \tilde{p}^+<2$, $\alpha>n(\frac{1}{\min\{1,\tilde{p}^-\}}-\frac{1}{2})$ and $r\in(\frac{n/\min\{1,\tilde{p}^-\}}{\alpha+n/2},1]$. Then there exists a positive constant $C$ such that, for any $\tau\in \mathbb{R}$ and $f\in H^{\vec{p}}_{L}(\rn)$,
\begin{align*}
\lf\|L^{i\tau}f\r\|_{H^{\vec{p}}_{L}(\rn)}\leq C\lf(1+|\tau|\r)^{n(\frac{1}{\min\{1,\tilde{p}^-\}}-\frac{r}{2})}\|f\|_{H^{\vec{p}}_{L}(\rn)}.
\end{align*}
\end{theorem}

The proofs of Theorems \ref{thms2} \ref{thms2x}, \ref{thms3} and \ref{thms3x} are similar to that of Theorem \ref{thms1}, thus we omit the details here.


\smallskip

{\bf Acknowledgements.}
X. Liu is supported by
the Gansu Province Education Science and Technology Innovation Project (Grant No. 224040), and
the Foundation for Innovative Fundamental Research Group Project of Gansu Province (Grant No. 25JRRA805).
W. Wang is supporeted the China Postdoctoral Science Foundation (No. 2024M754159), and the  Postdoctoral Fellowship Program of
CPSF (No. GZB20230961).
The authors would like to thank Prof. Dachun Yang and Prof. Sibei Yang for their valuable suggestions and insightful comments, which greatly contributed to the improvement of this paper.




\bigskip

\noindent Xiong Liu: School of Mathematics and Physics,
Gansu Center for Fundamental Research in Complex Systems Analysis and Control,
Lanzhou Jiaotong University, Lanzhou 730070, P. R. China

\medskip

\noindent
Wenhua Wang (Corresponding author):
 Institute for Advanced Study in Mathematics, Harbin Institute of Technology, Harbin 150001,
P. R. China

\medskip

\smallskip

\noindent{E-mails}:\\
\texttt{liuxmath@126.com} (Xiong Liu)  \\
\texttt{whwangmath@whu.edu.cn} (Wenhua Wang)
\bigskip \medskip

\end{document}